\documentclass[a4paper,12pt]{amsart}
\usepackage{amssymb}

\usepackage{xy}
\let\<\langle
\let\>\rangle
\DeclareMathOperator{\Id}{Id}
\DeclareMathOperator{\diag}{diag}
\let\tilde\widetilde
\let\hat\widehat
\renewcommand{\MR}[1]{}

\numberwithin{equation}{section}

\theoremstyle{plain}
\newtheorem{theorem}[equation]{Theorem}
\newtheorem{proposition}[equation]{Proposition}
\newtheorem{lemma}[equation]{Lemma}
\newtheorem{corollary}[equation]{Corollary}
\newtheorem*{claim*}{Claim}
\newtheorem{claim}{Claim}

\theoremstyle{definition}
\newtheorem{definition}[equation]{Definition}

\theoremstyle{remark}
\newtheorem{remark}[equation]{Remark}

\title[Frobenius manifolds with a tt*-structure]{Some constraints on Frobenius manifolds with a tt*-structure}

\begin{document}
\author{Jiezhu Lin}
\address{Guangzhou, China}
\email{ljzsailing@163.com}

\begin{abstract}
The article gives a necessary and sufficient condition for a
Frobenius manifold to be a CDV-structure. We show that there
exists a positive definite CDV-structure on any semi-simple
Frobenius manifold. We also compare three natural connections on a
CDV-structure and conclude that the underlying Hermitian manifold
of a non-trivial CDV-structure is not a K\"ahler manifold.
Finally, we compute the harmonic potential of a harmonic Frobenius
manifolds.
\end{abstract}


\thanks{The author thanks CMLS of Ecole Polytechnique in Palaiseau for good working
condition during her visit in May 2008, and especially professor
Claude Sabbah for his invitation, hospitality and enlightening
conversations. Her work is partly supported by
NKBRPC(2006CB805905), CMLS of Ecole Polytechnique and DMA of Ecole
Normale sup\'erieure.}

\maketitle

\setcounter{section}{-1}

\section{Introduction}\label{section0}

Cecotti and Vafa \cite{CV3} \cite{CVN} considered moduli spaces of
$N=2$ super-symmetric quantum field theories and introduced a
geometry on them which is governed by the tt*-equations. By the
work of K.~Saito and M.~Saito, it was previously known that the
base space of a semi-universal unfolding of a hypersurface
singularity can be equipped with the structure of what is now
called a Frobenius manifold, after \cite{Dubro}. By the work of
Cecotti and Vafa it can be equipped with tt* geometry if the
singularity is quasi-homogeneous. tt* geometry generalizes the
notion of variation of Hodge structures. Inspired by the papers
\cite{Dubro}, \cite{CV3} and \cite{CVN}, C.~Hertling combines
these two structures together into a structure which he calls a
\emph{CDV-structure}.

The purpose of the first part of this article is to give a
necessary and sufficient condition for a Frobenius manifold to be
a CDV-structure. This condition plays an important role in
constructing examples of CDV-structures, and in the existence
CDV-structures on any semi-simple Frobenius manifold. The purpose
of the second part of this article is to compare three natural
connections on a CDV-structure and conclude that the real
structure $\kappa$ cannot be flat and the underlying real $(1,
1)$-form of a nontrivial semi-simple CDV-structure can not be a
symplectic form. In particular, the underlying Hermitian manifold
of a nontrivial semi-simple positive CDV-structure cannot be a
K\"ahler manifold.

\subsubsection*{Acknowledgement}
The author is grateful to professor Claude Sabbah for pointing me
to the work of Claus Hertling, and for his patient helps,
encouragements, valuable suggestions and fruitful discussions. She
also expresses her gratitude to professor Juanxun Hu for his
patient helps, valuable suggestions and encouragements.

\section{Frobenius manifold and tt* geometry}\label{section1}

In this section we recall the notion of a Frobenius manifold and
CDV-structure. This will mainly serve to fix notation.

\subsection{Saito structure and Frobenius manifold structure}\label{subsection1a}
Frobenius manifolds were introduced and investigated by B.
Dubrovin as the axiomatization of a part of the rich mathematical
structure of the Topological Field Theory (TFT): cf. \cite{D,
Hert, Mani}

A Frobenius manifold (also called Frobenius structure on $M$) is a
quadruple $(M, \circ, g, e, \mathcal{E})$. Here $M$ is a manifold
in one of the standard categories ($C^\infty$, analytic, ...), $g$
is a metric on $M$ (that is, a symmetric, non-degenerate bilinear
form, also denoted by $\< \,,\, \>$), $\circ$ is a commutative and
associative product on $T_M$ and depends smoothly on $M$, such
that if $\nabla$ denotes the Levi-Civita connection of $g$, all
subject to the following conditions:
\begin{enumerate}
\item[a)] $\nabla$ is flat;

\item[b)] $g( X \circ Y, Z) = g( X, Y \circ Z)$, for any $X, Y, Z
\in TM$.

\item[c)] the unit vector field e is covariant constant w.r.t.
$\nabla$
\begin{eqnarray*}
\nabla e = 0;
\end{eqnarray*}

\item[d)] Let
\begin{equation*}
c(X, Y, Z):=g(X \circ Y, Z)
\end{equation*}
(a symmetric 3-tensor). We require the 4-tensor
\begin{equation*}
(\nabla_Z c)(U, V, W)
\end{equation*}
to be symmetric in the four vector fields $U, V, W, Z$.

\item[e)] A vector field $\mathcal{E}$ must be determined on $M$
such that
\begin{align}
\nabla (\nabla \mathcal{E}) &= 0;\label{nablanablaE}\\
\mathcal{L}_{\mathcal{E}}(\circ)&=\circ;\label{LEcirc}\\
\mathcal{L}_{\mathcal{E}} ( g)&=D \cdot g.\label{LEg}
\end{align}
\end{enumerate}

\begin{remark}\label{remark SABB}
In this definition, because the metric $g$ is flat and the unit
field $e$ is covariant constant w.r.t$. \nabla$, then \eqref{LEg}
implies \eqref{nablanablaE}.

Good reference is the last chapter in \cite{Sabb}.
\end{remark}

There are several equivalent ways to describe a Frobenius
structure. One way, called Saito structure, is recalled here:

\begin{definition}\label{deifinition: Saito structure}

Let $M$ be a complex analytic manifold of dimension~$m$. A Saito
structure on $M$ (without metric) consists of the following data:

\begin{enumerate}
\item[1)] a flat torsion free connection $ \nabla$ on the tangent
bundle $\mathcal{T}_M$;

\item[2)] a symmetric Higgs field $\Phi$ on the tangent bundle
$\mathcal{T}_M$, that is, $\Phi$ is an $\mathcal{O}_M$-linear map
$\Phi$: $\Theta_M \rightarrow \Omega_M^{1} \otimes \Theta_M$ such
that
\begin{equation*}
\Phi_X \Phi_Y = \Phi_Y \Phi_X;
\end{equation*}
\item[3)] two global sections (vector fields) $e$ and
$\mathcal{E}$ of $\Theta_M$, respectively called unit field and
Euler field of the structure.
\end{enumerate}

These data are subject to the following conditions:

\begin{enumerate}
\item[a)] the meromorphic connection $\widetilde{\nabla}$ on the
bundle $\pi^*\mathcal{T}_M$ on $\mathbb{P}^1 \times M$ defined by
the formula
\begin{eqnarray*}
\widetilde{\nabla}=\pi^* { \nabla} + \frac{\pi^* \Phi}{z} -
\Big(\frac{\Phi(\mathcal{E})}{z} + \nabla
\mathcal{E}\Big)\frac{dz}{z}
\end{eqnarray*}
is integrable;

\item[b)] the field $e$ is $ \nabla$-horizontal (i.e., $\nabla e =
0$) and satisfies $\Phi_e = - \Id$ (i.e., the product $\circ$
associated to $\Phi$ has $e$ as a unit field).
\end{enumerate}
\end{definition}

\begin{definition}\label{definition: Saito structure with metric}
Let $M$ be a complex analytic manifold of dimension~$m$. A Saito
structure on $M$ with metric consists of a Saito structure $(
\nabla, \Phi, e, \mathcal{E})$ and of a metric $ g$ on the tangent
bundle, satisfying the following properties:
\begin{enumerate}
\item[(1)] $ \nabla g = 0$ (hence $ \nabla$ is the Levi-Civita
connection of $ g$);

\item[(2)] $\Phi^*=\Phi$, i.e., for any local section $X$ of
$\Theta_M$, $\Phi^*_X=\Phi_X$, where $^*$ denotes the adjoint
w.r.t. $ g$;

\item[(3)] there exists a complex number $d \in \mathbb{C}$ such
that
\begin{equation*}
{\nabla}\mathcal{E}+({\nabla}\mathcal{E})^*= (2-d) \cdot \Id;
\end{equation*}
\end{enumerate}
\end{definition}

\begin{proposition}[\cite{D,Sabb}]\label{proposition3}
On any manifold $M$, there is an equivalence between a Saito
structure with metric and a Frobenius structure.
\end{proposition}

Locally, given a Frobenius manifold structure on open subset $U
\subset \mathbb{C}^{m}$, then we can find a function $F=F(t)$,
$t=(t^1, t^2,\dots, t^m)$, such that its third derivatives
$$C_{ijk}:=\frac{\partial F}{\partial t^i \partial t^j \partial
t^k}$$ satisfy the following equations
\begin{enumerate}
\item[1)] Normalization:
$$g_{ij}:=C_{1ij}$$ is a constant non-degenerate
matrix. Let
$$(g^{ij}):=(g_{ij})^{-1}$$

\item[2)] Associativity: the functions
$${C_{ij}}^k:= \sum_{l} C_{ijl}
\cdot g^{lk}$$ define a commutative and associative algebra on
$T_t M$ by
$$\partial_{t^i} \circ \partial_{t^j}:= \sum_k C_{ij}^k
\partial_{t^k}$$

\item[3)] Homogeneity: The function $F$ must be quasi-homogeneous,
i.e.,
\[
\mathcal{L}_{\mathcal{E}} F = d_F \cdot F + \text{quadratic
terms},
\]
where $\mathcal{E}=\sum_{i,j} (q_i^j t^i + r^j )\partial_{t^j}$,
and $d_F \in \mathbb{C}$.
\end{enumerate}

If the eigenvalues of $\nabla \mathcal{E}$ are simple, then the
Euler vector field can be reduced to the form
$$\mathcal{E} = \sum_i d_i t^i \partial_{t^i} + \sum_{j \mid d_j =0} r_j \partial_{t^j}.$$
where all $r_j$ are complex numbers, and all $d_i$ are the
eigenvalues of $\nabla \mathcal{E}$. Moreover, if $g(e, e)=0$, we
have
\begin{proposition}[\cite{D}]\label{proposition4}
Let $M$ be a Frobenius manifold. Assume that $g(e, e)=0$ and that
the eigenvalues of $\nabla \mathcal{E}$ are simple. Then by a
linear change of coordinates t$_i$ the matrix $g_{ij}$ can be
reduced to the anti-diagonal form
\begin{align}
g_{ij}=\delta_{i+j,m+1};\label{gij}\\
e=\partial_{t^1}.\label{et1}.
\end{align}
and in these coordinates
\begin{align}
F(t)= \frac{1}{2} (t^1)^2 t^m + \frac{1}{2} (t^1)^2 \sum_i t^i
t^{m-i+1} + f(t^2, t^3,\dots, t^m).\label{LieF}
\end{align}
for some function, the sum
$$d_i + d_{m-i+1}$$ does not depend on $i$, and
$$d_F = 2d_1 + d_m.$$
If the degrees are normalized in such a way that $d_1 =1$ then
they can be represented in the form
$$d_i = 1-q_i;\quad d_F=3-d,$$
where $q_1, q_2,\dots, q_m,d$ satisfy
$$q_1 = 0,\quad q_m = d,\quad q_i
+q_{m-i+1} =d.$$
\end{proposition}

So, under the assumption of Proposition \ref{proposition4}, we can
choose a flat holomorphic local coordinates $t^1, t^2,\dots, t^m$
of $M$ such that $g_{ij}=\delta_{i+j,m+1}$, $e=
\partial_{t^1}$ and
\begin{align}
\mathcal{E}=\sum_{i}d_{i}t^{i}\partial_{t^i}+ \sum_{i |d_i =0}
r^{i} \partial_{t^i};\label{E}\\
d_1 =1;\label{d1}\\
d_i + d_{m+1-i}=2-d;\label{did}\\
d_F = 3-d.\label{dFd}
\end{align}

\subsection{CV-structure}\label{subsection1b}
In \cite{Hert2} Hertling considers the notion of a CV-structure on
any $C^{\infty}$ vector bundle $K\rightarrow M$. He also considers
the notion of a CDV-structure on a manifold, which is a
CV-structure on the tangent bundle of a Frobenius manifold $M$
satisfying some compatibility conditions. We now recall these
structures and their basic properties.

\begin{definition}[\cite{Hert2}]\label{D structure}
Let $M$ be a complex analytic manifold. A $DC\tilde{C}$-structure
is a $C^{\infty}$ vector bundle $K \rightarrow M$ together with a
connection $D$ on it and two $C_M^{\infty}$-linear maps
\begin{align*}
C:C^{\infty}(K) &\longrightarrow \mathcal{A}_M^{1,0} \otimes
C^{\infty}(K);\\
\tilde{C} :C^{\infty}(K) &\longrightarrow \mathcal{A}_M^{0,1}
\otimes C^{\infty}(K)
\end{align*}
with the following properties. Let $D'$ and $D^{''}$ be the
$(1,0)$-part and the $(0,1)$-part of $D$, then
\begin{align}
(C+D^{''})^2=0, (\tilde{C}+D')^2=0; \label{holoC}\\
D'(C)=0, D^{''}(\tilde{C})=0; \label{tt1}\\
D'D^{''}+D^{''}D'=-(C \tilde{C}+\tilde{C} C). \label{tt2}
\end{align}
\end{definition}

\begin{remark}\label{remark tt}
\begin{enumerate}
\item[a)] In this definition, the equations \eqref{tt1} and
\eqref{tt2} are called tt*-equations in \cite{CV3, CV2, CV1}.

\item[b)] by \eqref{holoC}, we have a family flat $(0,
1)$-connections $D''+z\tilde C$ on $K$ for any $z\in\mathbb C$, so
$K$ comes equipped with a family of holomorphic structures.
\end{enumerate}
\end{remark}

\begin{definition}[\cite{Hert2}]\label{CV structure}
Let $M$ be a complex analytic manifold. A CV-structure is a
quadruple $(K\rightarrow M, D, C, \tilde{C},\kappa, h,
\mathcal{U}, \mathcal{Q})$ such that $(K\rightarrow M, D, C,
\tilde{C})$ is a $DC\tilde{C}$-structure, and the other objects
have the following properties:
\begin{enumerate}
\item[a)] $\kappa$ is a fiberwise $\mathbb{C}$-anti-linear
automorphism of $K$ as a $C^{\infty}$- bundle with
\begin{align}
\kappa^2&=\Id; \label{involution}\\
D(\kappa)&=0; \label{flatk}\\
\kappa C \kappa &= \tilde{C}. \label{CC}
\end{align}

\item[b)] $h$ is a Hermitian pseudo-metric on $K$; that is, it is
linear on the left, semi-linear on the right, non-degenerate, and
satisfies $h(b,a)= \overline{h(a,b)}$. It also has the three
properties:
\begin{align}
&\text{$h$ takes real values on $K_{\mathbb{R}}:=
\ker(\kappa - \Id) \subset K$}; \label{real h}\\
&h(C_X a, b)=h(a, \tilde{C}_{\overline{X}}b)\quad \text{for }a, b \in C^{\infty}(K),\ X \in \mathcal{T}_M^{1,0}; \label{duality}\\
&D(h)=0. \label{flath}
\end{align}

\item[c)] $\mathcal{U}$ and $\mathcal{Q}$ are
$C_M^{\infty}$-linear endomorphism of $K$ with
\begin{align}
[C, \mathcal{U}]=0; \label{CU}\\
D'(\mathcal{U})- [C, \mathcal{Q}]+C=0; \label{UCQ}\\
D^{''}(\mathcal{U})=0; \label{U}\\
D'(\mathcal{Q})+ [C,\kappa \mathcal{U} \kappa]=0; \label{QCU}\\
\kappa \mathcal{Q} \kappa + \mathcal{Q}=0; \label{Qkappa}\\
h(\mathcal{U} a, b)=h(a, \kappa \mathcal{U} \kappa b); \label{UU}\\
h(\mathcal{Q} a, b)=h(a, \mathcal{Q} b). \label{QQ}
\end{align}
\end{enumerate}
\end{definition}

\begin{remark}
Given any CV-structure $(K \rightarrow M, D, C, \tilde{C},\kappa,
h, \mathcal{U}, \mathcal{Q})$, if $h$ is positive, then $(K
\rightarrow M, D, C, \tilde{C},\kappa, h, \mathcal{U},
\mathcal{Q})$ is called \emph{positive CV-structure}, denote by
CV$\oplus$-structure.
\end{remark}

If we combine the Frobenius manifold structure and CV-structure
together, then we get the following structure.

\begin{definition}[\cite{Hert2}]\label{CDV-structure}
Let $M$ be a complex analytic manifold. A CDV-structure on $M$ is
a CV-structure $(\mathcal{T}_M^{(1, 0)} \rightarrow M, D, C,
\tilde{C},\kappa, h, \mathcal{U}, \mathcal{Q})$, together with a
Frobenius manifold structure $(M, \circ, e, \mathcal{E}, g)$ such
that $C_X Y=- X\circ Y$, $\mathcal{E}=\mathcal{U}(e)$, and
\begin{equation*}
\mathcal{Q}=D_{\mathcal{E}}-\mathcal{L}_{\mathcal{E}}-\frac{2-d}{2}
\Id,
\end{equation*}
for some $d \in \mathbb{R}$ and such that the following equivalent
conditions hold
\begin{align}
D_e - \mathcal{L}_e=0;\label{ee}\\
\Leftrightarrow D_e e=0;\label{dee}\\
\Leftrightarrow \mathcal{L}_e(h)=0 \Leftrightarrow \mathcal{L}_{\overline{e}}(h)=0;\label{eh}\\
\Leftrightarrow \mathcal{L}_e(\kappa)=0 \Leftrightarrow
\mathcal{L}_{\overline{e}}(\kappa)=0.\label{ekappa}
\end{align}
\end{definition}

Given a Frobenius manifold, giving a CDV-structure on it amounts
to giving a real structure on $\mathcal{T}_M^{(1, 0)}$ satisfying
some compatibility conditions made precise in the following
proposition.

\begin{proposition}[\cite{Hert2}]\label{CDV-structure2}
Let $M$ be a complex analytic manifold. A CDV-structure on $M$ is
a Frobenius manifold $(M, \circ, e, \mathcal{E}, g)$ together with
a real structure on $\mathcal{T}_M^{(1, 0)}$ given by a fiber-wise
$\mathbb{C}$-anti-linear involution $\kappa : \mathcal{T}_M^{(1,
0)} \rightarrow \mathcal{T}_M^{(1, 0)}$ such that the following
holds:

\begin{enumerate}
\item\label{CDV-structure21} Extend $g$ to the complex bundle
$\mathcal{T}_M^{(1, 0)}$. The form $h:=g(\cdot, \kappa \cdot)$ is
a Hermitian pseudo-metric and satisfies
\begin{align}
h(C_X Y, Z)&=h(Y, \kappa C_X \kappa Z ),\quad \text{for }X, Y, Z \in \mathcal{T}_M;\label{Ch2}\\
\mathcal{L}_e (h)&=0;\label{eh2}\\
\mathcal{L}_{\mathcal{E}-\overline{\mathcal{E}}}(h)&=0;\label{EEh2}
\end{align}
\item The metric connection $D$ on $\mathcal{T}_M^{(1, 0)}$ for
$h$ respects $\kappa$, i.e. $D(h)=0$ and $D(\kappa)=0$. \item The
number $d$ such that $\mathcal{L}_{\mathcal{E}}(g)= (2-d) \cdot g$
is real. \item Let $\mathcal{Q}$ be the endomorphism of $\mathcal
T_M^{(1,0)}$ as real analytic complex vector bundle defined by
$$\mathcal{Q}:= D_{\mathcal{E}}- \mathcal{L}_{\mathcal{E}}-\frac{2-d}{2}
\Id,$$ and let $\pi:\mathbb C\times M\to M$ be the projection.
Lift $D$ and $\mathcal{Q}$ canonically to $\pi^* \mathcal
T_M^{(1,0)}$. Then
$$\nabla^{CV}:= D+\frac{1}{z} C + z \kappa C \kappa + \Big(\frac{1}{z} \mathcal{U} -\mathcal{Q} -z \kappa \mathcal{U}
\kappa\Big)\frac{dz}{z}$$ is a flat meromorphic connection on
$\pi^* \mathcal T_M^{(1,0)} \vert_{\mathbb{C}^* \times M}$.
\end{enumerate}
\end{proposition}

\subsection{K\"ahler manifolds}\label{subsec:Kahler}
We recall classical results in order to fix notation (see e.g.\
\cite{CV}). Any complex analytic manifold $M$ comes equipped with
an almost complex manifold $(M_{\mathbb{R}}, J)$, where
$M_{\mathbb{R}}$ is the underlying real manifold of $M$, and $J$
is the complex structure induced by $i$. We get two complex vector
bundles on $M_{\mathbb{R}}$: on one hand, $(\mathcal{T}_M^{(1,
0)}, i)$ is a complex vector bundle on $M_{\mathbb{R}}$; on the
other hand, $(T_{M_{\mathbb{R}}}, J)$ is another vector bundle,
where $T_{M_{\mathbb{R}}}$ is the real tangent bundle of
$M_{\mathbb{R}}$. These two complex vector bundles are isomorphic.
Let us recall the isomorphism.

Let $z^1, z^2,\dots, z^m$ be a system of holomorphic local
coordinates on $M$, and set
$$z^j= x^j + i y^j,\quad \forall j.$$ Then $x^1,\dots,
x^m, y^1,\dots, y^m$ is a system of real local coordinates of
$M_{\mathbb{R}}$ and
\begin{equation*}
\partial_{z^j} = \frac{1}{2}(\partial_{x^j}- i \partial_{y^j}).
\end{equation*}
On $T_{M_{\mathbb{R}}}$, $J$ is determined by
\begin{equation*}
J \partial_{x^j}= \partial_{y^j},\quad J \partial_{y^j}= -
\partial_{x^j}, \forall j.
\end{equation*}
Define a map $\mathcal{R}e$ as follows:
$$\mathcal{R}e: (\mathcal{T}_M^{(1, 0)}, i)\rightarrow (T_{M_{\mathbb{R}}}, J),$$
$$\mathcal{R}e(\partial_{z^j}):= (\partial_{x^j}),\quad \mathcal{R}e(i \partial_{z^j}):= (\partial_{y^j}),\quad \forall j.$$ We note that
the map $\mathcal{R}e$ defined above is an isomorphism of real
vector bundles on $M_{\mathbb{R}}$, and satisfies
$$\mathcal{R}e \circ i = J \circ \mathcal{R}e,$$
where $\circ$ is the composition of endomorphisms. So it is an
isomorphism of complex vector bundles.

Moreover, $\mathcal{R}e$ induces a dual isomorphism
$$\mathcal{R}e^{*} :\Omega_{M_{\mathbb{R}}}^1 \rightarrow
\mathcal{A}_{M}^{(1,0)},\quad dx^j \mapsto dz^j,\; dy^j \mapsto -
i dz^j,\quad \forall j.$$ So we have an induced isomorphism from
$\mathcal{R}e$ and ${\mathcal{R}e^{*}}^{-1}$:
$$\widetilde{\mathcal{R}e}:= \mathcal{R}e \otimes {\mathcal{R}e^{*}}^{-1} :\mathcal{A}_{M}^{(1,0)}\otimes \mathcal{T}_{M}^{(1,0)} \rightarrow
\Omega_{M_{\mathbb{R}}}^1 \otimes T_{M_{\mathbb{R}}}.$$

Any Hermitian pseudo-metric $h$ on $M$ can be decomposed as $$h =
\hat{g}- i \hat{\omega},$$ where $\hat{g}$ (resp. $-\hat{\omega}$)
is the real part (resp.\ imaginary part) of $h$. Then $\hat{g}$ is
a Riemannian pseudo-metric(that is, a symmetric, non-degenerate
bilinear form) and $\hat{\omega}$ is a real $(1, 1)$-form on
$M_{\mathbb{R}}$. $(M_{\mathbb{R}}, \hat{\omega})$ is called a
symplectic manifold if $d \hat{\omega}=0$. $(M, h)$ is called
K\"ahler if $d \hat\omega=0$ and $h$ is positive-definite. The
Levi-Civita connection of $(T_{M_{\mathbb{R}}}, \hat{g})$ is
denoted by $\hat{\nabla}$, and the Chern connection of
$(\mathcal{T}_{M}^{(1, 0)}, h)$ by $D'$.

Theorem 3.13 in \cite{CV} gives a characterization of K\"ahler
metrics: $h$ is K\"ahler if and only if the Chern connection $D'$
and the Levi-Civita connection $\hat{\nabla}$ coincide on
$\mathcal{T}_{M}^{(1, 0)}$, identified with $T_{M_{\mathbb{R}}}$
via the map $\mathcal{R}e$. That means
\begin{equation}\label{eq:Re}
\widetilde{\mathcal{R}e}(D' V)= \hat{\nabla} \mathcal{R}e(V),\quad
\forall V\in \mathcal{T}_{M}^{(1, 0)}.
\end{equation}

\section{Main results}\label{subsection1c}
We first give a simple necessary and sufficient condition on a
real structure $\kappa$ to produce a CDV-structure on a Frobenius
manifold $M$.

\begin{theorem}\label{suff cond}
Let $(M, g, \circ, e, \mathcal{E})$ be a Frobenius manifold, and
let $\kappa$ be an anti-linear involution of
$\mathcal{T}_M^{(1,0)}$ such that $h(a, b):=g(a, \kappa b)$
satisfies $h(a, b)=\overline{h(b, a)}$. Let us denote by $D_h =
D'+ \overline{\partial}$ the Chern connection of~$h$ and by $d$
the real number such that $\mathcal{L}_{\mathcal{E}}(g)=(2-d)g$.
Let us set $\mathcal{Q}:= D'_{\mathcal{E}}-
\mathcal{L}_{\mathcal{E}} - \frac{2-d}{2} \cdot \Id$ and let us
define $\Phi$ by ${\Phi_X}{Y} := -X \circ Y$. We also denote by
$\mathcal{Q}^\dag$ and $\Phi^\dag$ the $h$-adjoints of
$\mathcal{Q}$ and $\Phi$.

Assume moreover that
\begin{align}
\mathcal{Q}&= \mathcal{Q}^{\dag};\label{Q}\\
D'(\Phi)&=0;\quad D'\overline{\partial}+ \overline{\partial}D'=
-(\Phi \wedge \Phi^{\dag}+\Phi^{\dag} \wedge
\Phi).\label{harmonic}
\end{align}
Then $(M, g, \circ, e, \mathcal{E}, \kappa)$ is a CDV-structure.
\end{theorem}

Clearly, the assumptions of Theorem \ref{suff cond} are necessary
to get a CDV-structure. When a Frobenius manifold is trivial (cf.
Definition \ref{def:trivial} below), we can find the discussion of
the Hermitian metric in Dubrovin's paper \cite{Dubro}. As a
consequence, we show that CDV$\oplus$-structures exist on all
semi-simple Frobenius manifolds by giving an explicit example:

\begin{theorem}\label{existence}
Let $(M, g, \circ, e, \mathcal{E})$ be a semi-simple Frobenius
manifold. Let $(u^{1}, u^{2},\dots, u^{m})$ be a system of
canonical local coordinates of $M$ and let $\eta$ be the
associated metric potential. Define a matrix $K$ of functions on
$M$ by
\begin{equation*}
K:= \diag\Big(\frac{|\eta_1|}{\eta_1},
\frac{|\eta_2|}{\eta_2},\dots, \frac{|\eta_m|}{\eta_m}\Big),
\end{equation*}
Let $\kappa: \mathcal{T}_M^{(1,0)}\rightarrow
\mathcal{T}_M^{(1,0)}$ be the $\mathbb{C}$ anti-linear
endomorphism  defined by
\begin{equation*}
\kappa \partial_{u^\alpha} = K_{\alpha\alpha} \partial_{u^\alpha}.
\end{equation*}
Then $(M, g, \circ, e, \mathcal{E}, \kappa)$ is a
CDV$\oplus$-structure on $M$ with $\mathcal{Q}=0$, and
\begin{equation*}
h= \diag({|\eta_1|}, {|\eta_2|},\dots, {|\eta_m|}).
\end{equation*}
Moreover the connection forms
$\omega_\alpha^\beta:=\sum_\gamma\partial h_{\alpha\gamma}\cdot
h^{\gamma\beta}$ for the Chern connection of $h$ are holomorphic
and the matrix $\omega=(\omega_\alpha^\beta)$ is given by
\begin{equation*}
\omega = \diag\Big(\frac{\partial \eta_1}{2\eta_1}, \frac{\partial
\eta_2}{2\eta_2},\dots, \frac{\partial \eta_m}{2\eta_m}\Big).
\end{equation*}
\end{theorem}

In the second part of this article we compare three natural
connections on a non-trivial CDV-structure. Given any
CDV-structure $(\mathcal{T}_M^{(1,0)}\rightarrow M, g, \circ, e,
\mathcal{E},\kappa)$ on a complex analytic manifold $M$, we have
three connections on the tangent bundle. The first one is the
Levi-Civita connection $\nabla$ for $g$, which is a holomorphic
connection, that we extend to a $(1,0)$-connection on
$\mathcal{T}_M^{(1,0)}$. The second one is the Chern connection
$D'$ with respect to $h$, where $h$ is defined above by $g$ and
$\kappa$, which is also a $(1,0)$-connection on
$\mathcal{T}_M^{(1,0)}$. The third one is the Levi-Civita
connection $\hat\nabla$ of the Riemannian pseudo-metric $\hat g$
(cf.\ \S\ref{subsec:Kahler}), that we consider as a
$(1,0)$-connection on $\mathcal{T}_M^{(1,0)}$ by using
\eqref{eq:Re} as the definition.

\begin{definition}\label{def:trivial}
A Frobenius manifold $(M, g, \circ, e, \mathcal{E})$ is said to be
trivial if, locally, the potential function is a polynomial of
degree three when expressed in some holomorphic $\nabla$-flat
local coordinates $t^1,\dots, t^m$.

A CDV-structure (resp.~a CDV$\oplus$-structure) is non-trivial if
the underlying Frobenius manifold is non-trivial.
\end{definition}

We have the following easy criterion for triviality.

\begin{lemma}
Let $(M, g, \circ, e, \mathcal{E})$ be a semi-simple Frobenius
manifold. Assume that any system of canonical local coordinates is
flat. Then the semi-simple Frobenius manifold is trivial and, for
any such system of coordinates, the potential $\eta$ of the metric
relative to this system, defined by
$g(\partial_{u^\alpha},\partial_{u^\alpha})=\partial\eta/\partial
u^\alpha$ ($\alpha=1,\dots,m$), is linear in the coordinates up to
an additive constant.\qed
\end{lemma}

We first show that $\nabla$ and $D'$ do not coincide on a
non-trivial semi-simple CDV-structure.

\begin{theorem}\label{flat k}
Let $(M, g, \circ, e, \mathcal{E})$ be a Frobenius manifold.
Suppose that there exists a $\mathbb{C}$-anti-linear involution
$\kappa$ such that $(M, g, \circ, e, \mathcal{E},\kappa)$ is a
CDV-structure.
\begin{enumerate}
\item\label{flat k1} The following properties are equivalent:
\begin{enumerate}
\item $\kappa$ is $\nabla$-flat, \item $D'=\nabla$.
\end{enumerate}
\item\label{flat k2} If these properties are satisfied and $(M, g,
\circ, e, \mathcal{E})$ is semi-simple, then any canonical local
coordinate system $(u^{1}, u^{2},\dots, u^{m})$ of $(M, g, \circ,
e, \mathcal{E})$ is $\nabla$-flat and $\mathcal Q=0$. In
particular, the semi-simple Frobenius manifold is trivial.
Moreover, if $(M, g, \circ, e, \mathcal{E},\kappa)$ is a
CDV$\oplus$-structure, in any such system of coordinates, the
matrices of $\kappa$ and $h$ are expressed as in Theorem
\ref{existence}.
\end{enumerate}
\end{theorem}

Let us check that \ref{flat k}\eqref{flat k1} holds. Because
$\kappa$ defines a CDV-structure, we have $D(\kappa)=0$, that is,
$D'(\kappa)=0$ and $\overline\partial(\kappa)=0$.

If $D'= \nabla$, we have
\begin{equation*}
(\nabla + \overline{\partial})(\kappa) = (D'+\overline{\partial})
(\kappa) = 0.
\end{equation*}
Conversely, if $\kappa$ is $\nabla$-flat, then $(\nabla +
\overline{\partial})(\kappa) = 0$ and
\begin{equation*}
\nabla \kappa = \kappa \overline{\partial} = D' \kappa
\end{equation*}
However, $\kappa$ is an involution of $\mathcal{T}_M^{(1, 0)}$,
hence $$D'=\nabla.$$

\begin{corollary}\label{D nabla}
For any non-trivial semi-simple CDV-structure, we have $D' \neq
\nabla$.\qed
\end{corollary}

\begin{corollary}\label{have torsion}
For any non-trivial semi-simple CDV-structure, the Chern
connection $D'$ is not torsion-free.
\end{corollary}

\begin{proof}
Assume that $D^{'}$ is torsion-free. Because of the relation $D^{'}(g)=\nobreak0$,
which is deduced from relations \eqref{flatk} and \eqref{flath} in
the definition of a CDV-structure, we conclude that $D^{'}$ is the
Levi-Civita connection of~$g$. So we have $D^{'}=\nabla$, in contradiction with corollary \ref{D nabla}.
\end{proof}

We will also compare $D'$ and $\nabla$ with the Levi-Civita
connection $\hat\nabla$ of the Riemannian pseudo-metric $\hat g$
(cf.\ \S\ref{subsec:Kahler}).

\begin{theorem}\label{non Kahler}
Let $(M, g, \circ, e, \mathcal{E}, \kappa)$ be a non-trivial,
semi-simple CDV-structure. Then
\begin{enumerate}
\item\label{non symp} $\hat{\omega}$ is not a symplectic form on
$M_{\mathbb{R}}$. Moreover, if $(M, g, \circ, e, \mathcal{E},
\kappa)$ is a CDV$\oplus$-structure, then the underlying Hermitian
manifold $(M, h)$ is not a K\"ahler manifold.
\item\label{comparison} the three connections $\nabla,
D',\hat{\nabla}$ are pairwise distinct (after the identification
of $\mathcal{T}_M^{(1,0)}$ with $T_{M_{\mathbb R}}$ via
$\mathcal{R}e$).
\end{enumerate}
\end{theorem}

In \cite{TA}, Atsushi Takahashi shows that the matrix $h$ of a
CDV$\oplus$-structure of dimension two with $d \neq 0$ is diagonal
when expressed in holomorphic $\nabla$-flat local coordinates. In
dimension bigger than two, we show that the opposite conclusion
holds.

\begin{corollary}\label{connections2}
Let $(M, g, \circ, e, \mathcal{E}, \kappa)$ be a non-trivial,
semi-simple CDV-structure. If $\dim_{\mathbb{C}}M \geq 3$, then
for any $\nabla$-flat holomorphic local coordinates $t^1, t^2,
\dots, t^m$, the matrix $(h_{ij})$ cannot be diagonal.
\end{corollary}

\section{Proof of the theorems}\label{section2}

\subsection{Necessary and sufficient condition: proof of Theorem \ref{suff cond}}\label{subsection2a}

In this subsection, we will use a system of holomorphic
$\nabla$-flat local coordinates $t^1, t^2,\dots,t^m$ of the
Frobenius manifold. We will then use the following notations:
\begin{align}
\kappa (\partial_{t^{i}})&= \sum_{k} K_{ik}\partial_{t^{k}},\label{eq:Mkappa}\\
\Phi_{\partial_{t^{i}}} (\partial_{t^{j}})&= - \sum_{k}
{C^{(i)}}_{j}^{k}
\partial_{t^{k}},\label{eq:MPhi}\\
\Phi^{\dag}_{ \overline \partial_{t^{i}}} (\partial_{t^{j}})&= -
\sum_{k} \tilde{C^{(i)}}_{j}^{k}\label{eq:MPhidag}
\partial_{t^{k}}.
\end{align}
If we define ${C_{ij}}^{k}$ by $$-\Phi _{\partial_{t^i}}
(\partial_{t^j}) = \sum_k {C_{ij}}^{k}
\partial_{t^k},$$ then we have
\begin{equation*}
{C^{(i)}}_{j}^{k} = {C_{ij}}^{k}.
\end{equation*}
Because of $h(X, Y)= g(X, \kappa Y)$ and $\Phi^{*}=\Phi$, we have,
for all $X,Y$,
\begin{eqnarray*}
h(X, \Phi^{\dag}Y)
& = &  h(\Phi X, Y) \\
& = & g(\Phi X, \kappa Y) \\
& = & h(X, \kappa \Phi \kappa Y),
\end{eqnarray*}
that is, $$\Phi^{\dag}= \kappa \Phi \kappa.$$ This is expressed by
$\tilde{C^{(i)}}= \overline{K} \cdot\nobreak \overline{C^{(i)}}
\cdot\nobreak K$, i.e.,
\begin{equation*}
\tilde{C^{(i)}}_{j}^{k}=\sum_{p,q}\overline{K}_{jp} \cdot
\overline{{C_{ip}}^{q}} \cdot K_{qk},
\end{equation*}

Let $D_h$ be the Chern connection of $h$ and let
$\omega_i^j:=\sum_k\partial h_{ik}\cdot h^{kj}$ be the connection
forms for $D'$ in the local holomorphic $\nabla$-flat coordinates
$t^i$.

\begin{proof}[Proof of Theorem \ref{suff cond}]
We just need to prove that
\[
(\mathcal{T}_M^{(1,0)} \rightarrow M, D_h, \Phi, \Phi^{\dag},
\kappa, h, \mathcal{U}:= \mathcal{E} \circ, \mathcal{Q})
\]
is a CV-structure and
\begin{equation*}
D'_{e}e=0.
\end{equation*}
Firstly, we will prove that $$D(\kappa)=0,$$ which is given by the
following two lemmas.

\begin{lemma}\label{lemma A}
If $h(a, b)=g(a, \kappa b)$ and $\kappa^2 = \Id$ hold, then we
have
\begin{equation*}
h^{-1}=\overline{g}^{-1} \cdot \overline{h} \cdot g^{-1}.
\end{equation*}
\end{lemma}

\begin{proof}[Proof of Lemma \ref{lemma A}]
If $h(a, b)=g(a, \kappa b)$, then
\begin{equation*}
h_{ij}= \sum_{k}K_{jk} \cdot g_{ki},
\end{equation*}
where $K$ is the matrix of $\kappa$ given by \eqref{eq:Mkappa} and
$h_{ij}=h(\partial_{t^i}, \partial_{t^i})$, that is,
\begin{equation*}
h^{t}=K \cdot g.
\end{equation*}
But $h$ satisfies $h(Y, X)= \overline{h(X, Y)}$, so
\begin{equation*}
\overline{h}=h^{t}=K \cdot g,
\end{equation*}
i.e.,
\begin{equation*}
K=\overline{h} \cdot g^{-1},
\end{equation*}
and thus
\begin{equation*}
h^{-1} = \overline{g}^{-1} \cdot \overline{K}^{-1}.
\end{equation*}
Now we will compute $\overline{K}^{-1}$.
\begin{equation*}
\partial_{t^{i}}= \kappa \kappa (\partial_{t^{i}})= \kappa \Big(\sum_{k} K_{ik} \cdot\partial_{t^{k}}\Big)= \sum_{k} \overline{K_{ik}}
\kappa(\partial_{t^{k}})= \sum_{k,l} \overline{K_{ik}} \cdot
K_{kl} \cdot \partial_{t^{l}},
\end{equation*}
that is
\begin{equation*}
\sum_{k,l} \overline{K_{ik}} \cdot K_{kl} = \delta_{i}^{l}.
\end{equation*}
So we have
\begin{equation*}
\overline{K} \cdot K = K \cdot \overline{K}= I_{m \times m}.
\end{equation*}
i.e.,
\begin{equation*}
\overline{K}^{-1}= K.
\end{equation*}
Therefore,
\begin{equation*}
h^{-1} = \overline{g}^{-1} \cdot \overline{K}^{-1} =
\overline{g}^{-1} \cdot K = \overline{g}^{-1} \cdot \overline{h}
\cdot g^{-1}.\qedhere
\end{equation*}
\end{proof}

\begin{lemma}\label{lemma 2th C}
Under the assumptions of Theorem \ref{suff cond}, we have
$$D_h(\kappa)=0.$$
\end{lemma}

\begin{proof}[Proof of Lemma \ref{lemma 2th C}]
By definition $D_h (\kappa)=0$ is equivalent to
\begin{equation*}
D'_X \kappa = \kappa \overline{\partial}_{\overline{X}},\quad
\forall X \in \Theta_M
\end{equation*}
This is equivalent to
\begin{equation*}
D'_{\partial_{t^i}} \kappa (\partial_{t^j}) = \kappa
\overline{\partial}_{\overline{\partial_{t^i}}}
(\partial_{t^j})\quad\forall i,j.
\end{equation*}
Clearly, the right-hand term is zero, hence proving
$D_h(\kappa)=0$ amounts to proving
\begin{equation}\label{eq:M}
\partial K + K\cdot\omega = 0.
\end{equation}

From Lemma \ref{lemma A}, we get
\begin{eqnarray*}
\omega
& = & \partial h \cdot h^{-1} \\
& = & (\partial \overline{K} \cdot \overline{g}) \cdot (\overline{g}^{-1} \cdot \overline{h} \cdot g^{-1}) \\
& = & \partial \overline{K} \cdot \overline{h} \cdot g^{-1} \\
& = & \partial K^{-1} \cdot \overline{h} \cdot g^{-1} \\
& = & - K^{-1} \cdot \partial K \cdot K^{-1} \cdot \overline{h} \cdot g^{-1} \\
& = & - \overline{K} \cdot \partial K \cdot \overline{K} \cdot \overline{h} \cdot g^{-1} \\
& = & - \overline{K} \cdot \partial K \\
& = & - K^{-1} \cdot \partial K,
\end{eqnarray*}
and thus,
$$\omega + K^{-1} \cdot \partial K =0,$$
which gives \eqref{eq:M}.
\end{proof}

Let us continue the proof of Theorem \ref{suff cond}. Having
proved $D(\kappa)=0$, we obtain
$$(D')^2 \kappa = D'\kappa \overline{\partial} = \kappa \overline{\partial}^2
=0,$$ and since $\kappa$ is an involution, we deduce $$(D')^2=0.$$

So, together with the assumption \eqref{harmonic}, we conclude
that $(D_{h}+ \Phi + \Phi^{\dag})^{2}=0$. Hence
$(\mathcal{T}_M^{(1,0)} \rightarrow M, D_h, \Phi, \Phi^{\dag})$ is
a $(DC\tilde{C})$-structure.

The relations (\ref{involution}), (\ref{CC}), (\ref{real h}),
(\ref{duality}), (\ref{flath}), (\ref{CU}) and (\ref{U}) hold
obviously.

The relation (\ref{UCQ}) follows from $D'(\Phi)=0$. In fact, if
$D'(\Phi)=0$ holds, we have $D'_{X}(\Phi_{\mathcal{E}})-
D'_{\mathcal{E}}(\Phi_{X})-\Phi_{[X, \mathcal{E}]} =0$ for any
$X\in\Theta_M$. So $D'(\mathcal{U})-[\Phi,\mathcal{Q}]+\Phi=0$ is
equivalent to $\mathcal{L}_{\mathcal{E}}(\circ)=\circ$, which is
part of the assumption that $M$ is a Frobenius manifold.

We will show that relation (\ref{QCU}) can be deduced from the
assumption $\mathcal{Q}= \mathcal{Q}^{\dag}$ and \eqref{harmonic}.

\begin{lemma}\label{lemma J}
Under the assumptions of Theorem \ref{suff cond}, we have
\begin{equation*}
D'(\mathcal{Q})+[\Phi,\kappa \mathcal{U} \kappa]=0
\end{equation*}
\end{lemma}

\begin{proof}[Proof of Lemma \ref{lemma J}]
We first show that, under the conditions of Theorem \ref{suff
cond},
\begin{equation}\label{eq:DXQ1}
-[\Phi,\kappa \mathcal{U} \kappa]=[\overline{\partial}_{\overline{\mathcal{E}}},D'_X]+(D'+\overline{\partial})_{[X,
\overline{\mathcal{E}}]},\quad\forall X\in\Theta_M.
\end{equation}
In fact, $D' \overline{\partial}+\overline{\partial}D' = -(\Phi
\wedge \Phi^{\dagger} + \Phi^{\dagger} \wedge \Phi )$ holds, so we
have
\begin{equation*}
(D' \overline{\partial}+\overline{\partial}D')(X,\overline{Y})=
-(\Phi \wedge \Phi^{\dagger} + \Phi^{\dagger} \wedge \Phi
)(X,\overline{Y}),\quad \forall X, Y \in \mathcal{T}_M^{(1,0)}.
\end{equation*}
That is,
\begin{equation*}
[D'_X,\overline{\partial}_{\overline{Y}}]-(D'+\overline{\partial})_{[X, \overline{Y}]}=-[\Phi_X,\Phi^{\dag}_{\overline{Y}}].
\end{equation*}
Take $Y= \mathcal{E}$, then
\begin{equation*}
[D'_{X},\overline{\partial}_{\overline{\mathcal{E}}}]-(D'+\overline{\partial})_{[X,\overline{\mathcal{E}}]}=-[\Phi_{X},\Phi^{\dag}_{\overline{\mathcal{E}}}]=[\Phi_{X},\kappa
\mathcal{U} \kappa],\quad\forall X \in \mathcal{T}_M^{(1,0)}.
\end{equation*}
This relation is linear with respect to $X$, so it holds for all
$X \in \mathcal{T}_M^{(1,0)}$ if and only if it holds for all
$X\in\Theta_M$, hence \eqref{eq:DXQ1}.

Now we will prove the lemma by proving
\begin{equation}\label{eq:DXQ}
D'(\mathcal{Q})=[\overline{\partial}_{\overline{\mathcal{E}}},D'_X]+(D'+\overline{\partial})_{[X,
\overline{\mathcal{E}}]},\quad\forall X\in\Theta_M.
\end{equation}
Under the assumptions of the
theorem, we have $\mathcal{Q}^{\dag}=\mathcal{Q}$, that is,
\begin{equation*}
h(\mathcal{Q}X, Y)=h(X, \mathcal{Q} Y),\quad \forall X, Y \in
\Theta_M.
\end{equation*}
So $\forall W \in \Theta_M$ we have
\begin{equation*}
\partial_{W} h(\mathcal{Q}X, Y) = h(D'_W \mathcal{Q} X, Y),
\end{equation*}
\begin{equation*}
\partial_{W} h(X,\mathcal{Q} Y) = h(D'_W X, \mathcal{Q} Y)+ h(X, \overline{\partial}_{\overline{W}} \mathcal{Q} Y),
\end{equation*}
So we get
\begin{equation*}
h(D'_W \mathcal{Q} X, Y)- h(D'_W X, \mathcal{Q} Y)= h(X,
\overline{\partial}_{\overline{W}} \mathcal{Q} Y),
\end{equation*}
i.e.,
\begin{equation}\label{WQD}
h(D'_W \mathcal{Q} X - \mathcal{Q} D'_W X, Y)= h(X,
\overline{\partial}_{\overline{W}} \mathcal{Q} Y)= h(X,
\overline{\partial}_{\overline{W}} D'_{\mathcal{E}} Y),
\end{equation}

\begin{claim}
$h(X, \overline{\partial}_{\overline{W}} D'_{\mathcal{E}} Y) =
h(\overline{\partial}_{\overline{\mathcal{E}}} D'_{W} X, Y)$,
$\forall X, Y, W \in \Theta_M$.
\end{claim}

In fact,
$$\overline{\partial}_{\overline{\mathcal{E}}} h(X, Y) =
h (X, D'_{\mathcal{E}} Y)$$ So
\begin{eqnarray*}
\partial_{W} \overline{\partial}_{\overline{\mathcal{E}}} h(X, Y)
& = & \partial_{W}h (X, D'_{\mathcal{E}} Y) \\
& = & h(D'_{W} X, D'_{\mathcal{E}} Y)+h(X,
\overline{\partial}_{\overline{W}} D'_{\mathcal{E}} Y )
\end{eqnarray*}
Similarly, we have
\begin{eqnarray*}
\overline{\partial}_{\overline{\mathcal{E}}} \partial_{W} h(X, Y)
& = & \overline{\partial}_{\overline{\mathcal{E}}} h (D'_W X, Y) \\
& = & h(D'_{W} X, D'_{\mathcal{E}}
Y)+h(\overline{\partial}_{\overline{\mathcal{E}}} D'_{W} X, Y )
\end{eqnarray*}
Since $W \in \Theta_M$, we get
\begin{equation*}
\partial_{W} \overline{\partial}_{\overline{\mathcal{E}}} h(X, Y)
=\overline{\partial}_{\overline{\mathcal{E}}} \partial_{W} h(X,Y),
\end{equation*}
hence
\begin{equation}\label{WEEW}
h(X, \overline{\partial}_{\overline{W}} D'_{\mathcal{E}} Y ) =
h(\overline{\partial}_{\overline{\mathcal{E}}} D'_{W} X, Y ).
\end{equation}
The relations (\ref{WQD}) and (\ref{WEEW}) imply
\begin{equation*}
h(D'_W (\mathcal{Q}) X, Y)=
h(\overline{\partial}_{\overline{\mathcal{E}}} D'_{W} X, Y )=
h([\overline{\partial}_{\overline{\mathcal{E}}}, D'_{W} ] X, Y
),\quad\forall X, Y, Y \in \Theta_M.
\end{equation*}
Because $h$ is non-degenerate,
\begin{equation*}
D'_W (\mathcal{Q})= [\overline{\partial}_{\overline{\mathcal{E}}},
D'_{W} ] = [\overline{\partial}_{\overline{\mathcal{E}}}, D'_{W}
]+ (D'+ \overline{\partial})_{[W,
\overline{\mathcal{E}}]},\quad\forall W \in \Theta_M,
\end{equation*}
hence \eqref{eq:DXQ}, and this ends the proof of Lemma \ref{lemma
J}.
\end{proof}
Let us continue to prove the theorem \ref{suff cond}.

\begin{claim}\label{claim0}
$D_h(\kappa)=0 \Leftrightarrow D'(g)=0$.
\end{claim}
We just prove $D_h(\kappa)=0 \Rightarrow D'(g)=0$, the other
direction holds similarly.

Since $D$ is the Chern connection of $h$, we have
\begin{equation*}
\partial h(X, Y) = h(D' X, Y) + h(X, \overline{\partial}Y),
\forall X, Y \in \mathcal{T}_M^{(1, 0)}.
\end{equation*}
If $D_h(\kappa)=0$, by $h(X, Y)= g(X, \kappa Y)$ and $\kappa^2 =
\Id$ we get
\begin{eqnarray*}
\partial g(X, Y) & = & \partial h(X, \kappa Y)\\
& = &  h(D' X, \kappa Y) + h(X, \overline{\partial}\kappa Y) \\
& = & h(D' X, \kappa Y) + h(X, \kappa D' Y) \\
& = & g(D' X, Y) + g(X, D' Y),
\end{eqnarray*}
i.e. $$D'(g)=0.\eqno\qed$$

\begin{claim}\label{claim01}
Under the condition $D'(g)=0$, $$\mathcal{Q}+ \mathcal{Q}^* =0
\Leftrightarrow \mathcal{L}_{\mathcal{E}}(g)= (2-d) \cdot g.$$
\end{claim}

This will prove that $\mathcal{Q}+ \mathcal{Q}^* =0$ holds under
the assumption of Theorem \ref{suff cond}, according to Lemma
\ref{lemma 2th C} and Claim \ref{claim0}, since the relation
$\mathcal{L}_{\mathcal{E}}(g)= (2-d) \cdot g$ is included in the
definition a Frobenius manifold.

\begin{proof}[Proof of Claim \ref{claim01}]
In fact, by definition, $\mathcal{Q}+ \mathcal{Q}^* =0$ is
equivalent to
$$g(\mathcal{Q}X, Y) + g(X, \mathcal{Q} Y) =0.$$
Computing the left hand side of the above relation, we get
\begin{align*}
g(\mathcal{Q}X, Y)&+ g(X, \mathcal{Q} Y) = g(D'_{\mathcal{E}} X,
Y) - g(\mathcal{L}_{\mathcal{E}}X, Y)-
\frac{2-d}{2} \cdot g(X, Y) \\
&\hspace*{2.4cm} +  g( X, D'_{\mathcal{E}} Y) - g(X, \mathcal{L}_{\mathcal{E}} Y)- \frac{2-d}{2} \cdot g(X, Y) \\
& =  \mathcal{E}g( X, Y) - g(\mathcal{L}_{\mathcal{E}}X, Y)-  g(X, \mathcal{L}_{\mathcal{E}} Y)-(2-d)\cdot g(X, Y)\\
& = \mathcal{L}_{\mathcal{E}}(g)(X, Y)- (2-d)\cdot g(X, Y)\\
& =  [\mathcal{L}_{\mathcal{E}}(g)- (2-d)\cdot g](X, Y).\qedhere
\end{align*}
\end{proof}

However $\mathcal{Q}+ \mathcal{Q}^* =0$ together with the
assumption $\mathcal{Q}= \mathcal{Q}^{\dag}$ imply
\begin{equation*}
\mathcal{Q} + \kappa \mathcal{Q} \kappa =0.
\end{equation*}
The relation $\mathcal{U}^{\dag} = \kappa \mathcal{U} \kappa$
holds because $\Phi^{\dag} = \kappa \Phi \kappa$. Hence
$(\mathcal{T}_M^{(1,0)} \rightarrow M, D_h, \Phi, \Phi^{\dag},
\kappa, h, \mathcal{U}, \mathcal{Q})$ is a CV-structure.

It remains to prove $D_{e}e=0$. As $e$ is holomorphic, it is
enough to prove
\begin{equation}\label{eq:Dee}
D'_e e=0.
\end{equation}
By the assumption of Theorem \ref{suff cond}, we have
$D'(\Phi)=0$, and since $\Phi_{e}=-\Id$ this implies
\begin{equation*}
D'_{e}(\Phi_{\partial_{t^i}})=
D'_{\partial_{t^i}}(\Phi_{e})=0,\quad\forall i
\end{equation*}
and therefore,
\begin{equation*}
D'_{e}(\Phi_{\partial_{t^i}})(e)=0\quad\forall i.
\end{equation*}
Computing the above relation we get
\begin{align*}
D'_{e}(\Phi_{\partial_{t^i}})(e)
& =  D'_{e}(\Phi_{\partial_{t^i}}e)- \Phi_{\partial_{t^i}}D'_{e}e \\
& =  -D'_{e} \partial_{t^i} + \partial_{t^i} \circ D'_{e}e \\
& =  0.
\end{align*}
i.e.,
\begin{equation*}
D'_{e} \partial_{t^i} = \partial_{t^i} \circ D'_{e}e, \quad\forall
i.
\end{equation*}
On the other hand, $D'(g)=0$ holds as a consequence of Lemma
\ref{lemma 2th C} and Claim \ref{claim0}, hence
\begin{equation*}
\partial g (X, Y)=g(D'X,Y)+g(X, D'Y),\quad \forall X, Y \in
\mathcal{T}_M^{(1,0)},
\end{equation*}
and therefore
\begin{equation*}
\partial_e g (e, Y)=g(D'_e e,Y)+g(e, D'_e Y),\quad \forall Y \in
\mathcal{T}_M^{(1,0)}.
\end{equation*}
Take $Y= \partial_{t^i}$. Then
\[
g(e, D'_e \partial_{t^i})=g(e, \partial_{t^i} \circ D'_e
e)=g(\partial_{t^i} \circ e, D'_e e)=g(D'_e e,\partial_{t^i}).
\]
As a consequence, since $e$ is holomorphic and flat,
\[
0=e g(e, \partial_{t^i})=g(D'_e e,\partial_{t^i})+g(e, D'_e
\partial_{t^i})=2 g(D'_e e,\partial_{t^i})\quad\forall i,
\]
giving thus \eqref{eq:Dee}. This ends the proof of Theorem
\ref{suff cond}.
\end{proof}

\subsection{Existence of a CDV$\oplus$-structure: proof of Theorem \ref{existence}}\label{subsection2d}

Let $(u^1, u^2,\dots, u^m)$ be a system of canonical local
coordinates of $M$. We will denote $e_\alpha=\partial_{u^\alpha}$
for $\alpha=1,\dots,m$. The matrix $K$ defined in Theorem
\ref{existence} obviously satisfies $$\overline{K}\cdot K=
I_{m\times m}.$$ Therefore, the associated anti-linear
endomorphism $\kappa$ is an involution of $\mathcal{T}_M^{(1,
0)}$. We will check this $\kappa$ together with the Frobenius
manifold structure define a CDV$\oplus$-structure on $M$.

Let $h$ be the sesquilinear form associated to $\kappa$ and $g$ as
in Proposition \ref{CDV-structure2}\eqref{CDV-structure21}. Then
\begin{equation}\label{hii}
(h_{\alpha \beta})_{m \times m}= \diag(|\eta_1|, |\eta_2|,\dots,
|\eta_m|).
\end{equation}
Since $g$ is non-degenerate, $\eta_\alpha:=g(e_\alpha,e_\alpha)$
does not vanish and $h$ is a Hermitian metric on $M$. Let $D'$ be
the Chern connection of $h$. Then the matrix $\omega$ of
connection forms of $D'$ satisfies
\begin{equation*}
\omega= \diag(\partial \log{|\eta_1|}, \partial \log{|\eta_2|},
\cdots, \partial \log{|\eta_m|}).
\end{equation*}
By a straightforward computation we get
\begin{equation}\label{omegaii}
\omega= \diag\Big(\frac{\partial \eta_1}{2\eta_1}, \frac{\partial
\eta_2}{2\eta_2},\dots, \frac{\partial \eta_m}{2\eta_m}\Big).
\end{equation}
So all $\omega_{\alpha}^{\alpha}$ are holomorphic $1$-forms.

\begin{claim*}
The Chern connection $D'$ of $h$ defined by \eqref{hii} satisfies
\eqref{harmonic}.
\end{claim*}

The relation $D'(\Phi)=0$ is a consequence the following lemma.

\begin{lemma}\label{canoDphi}
Let $M$ be a semi-simple Frobenius manifold $M$ of dimension $m$
and let $h$ be a non-degenerate sesquilinear form on $M$, with
associated Chern connection $D'$. Let $(u^1,\dots,u^m)$ be a local
system of canonical coordinates on $M$ and let
$\omega=(\omega_\alpha^\beta)$ be the connection matrix of $h$ in
these coordinates. Then $D'(\Phi)=0$ is equivalent to
\begin{itemize}
\item $\omega_\alpha^\beta(e_\alpha +e_\beta)=0$,
$\forall\alpha\neq \beta$, if $m=2$, \item and, if $m\geq3$, to
\[
\begin{cases}
\omega_\alpha^\beta(e_\gamma)=0,& \forall \gamma\neq \alpha,\, \gamma\neq \beta,\,\alpha\neq \beta,\\
\omega_\alpha^\beta(e_\alpha +e_\beta)=0,& \forall\alpha\neq
\beta.
\end{cases}
\]
\end{itemize}
\end{lemma}

\begin{proof}[Proof of Lemma \ref{canoDphi}]
By definition $D'(\Phi)=0$ is equivalent to
\begin{equation*}
D'_{e_\alpha}(\Phi_{e_\beta})= D'_{e_\beta}(\Phi_{e_\alpha}),
\quad\forall \alpha, \beta,
\end{equation*}
that is, to
\begin{equation}\label{dphi}
D'_{e_\alpha}(\Phi_{e_\beta})(e_\gamma)=
D'_{e_\beta}(\Phi_{e_\alpha})(e_\gamma), \quad\forall \alpha,
\beta, \gamma.
\end{equation}
Assume that $\alpha,\beta,\gamma$ are pairwise distinct. Then,
because $(u^\alpha)$ are canonical, $\Phi_{e_\alpha}(e_\gamma)=0$
and $\Phi_{e_\beta}(e_\gamma)=0$,  so
\begin{eqnarray*}
D'_{e_\alpha}(\Phi_{e_\beta})(e_\gamma)-
D'_{e_\beta}(\Phi_{e_\alpha})(e_\gamma)
& = & e_\beta \circ D'_{e_\alpha} e_\gamma - e_\alpha \circ D'_{e_\beta} e_\gamma \\
& = & \omega_{\gamma}^\beta(e_\alpha) e_\beta -
\omega_{\gamma}^\alpha(e_\beta) e_\alpha,
\end{eqnarray*}
Hence, for any such $\alpha,\beta,\gamma$,
$D'_{e_\alpha}(\Phi_{e_\beta})(e_\gamma)=
D'_{e_\beta}(\Phi_{e_\alpha})(e_\gamma)$ is equivalent to
$$\omega_{\gamma}^\beta(e_\alpha)=0.$$

Assume now $\alpha\neq\beta$ and take $\gamma=\alpha$ in the
relation (\ref{dphi}). Then
\begin{eqnarray*}
D'_{e_\alpha}(\Phi_{e_\beta})(e_\alpha)-
D'_{e_\beta}(\Phi_{e_\alpha})(e_\alpha)
& = & e_\beta \circ D'_{e_\alpha} e_\alpha - e_\alpha \circ D'_{e_\beta} e_\alpha + D'_{e_\beta} e_\alpha \\
& = & \omega_{\alpha}^\beta(e_\alpha) e_\beta - \omega_{\alpha}^{\alpha}(e_\beta) e_\beta + \sum_\gamma \omega_{\alpha}^{\gamma}(e_\beta) e_\gamma \\
& = & \omega_{\alpha}^\beta(e_\alpha) e_\beta + \sum_{\gamma\neq \alpha}\omega_{\alpha}^{\gamma}(e_\beta)e_\gamma \\
& = & \omega_{\alpha}^\beta(e_\alpha) e_\beta +
\omega_{\alpha}^\beta(e_\beta) e_\beta,
\end{eqnarray*}
The last equality holds because
$\omega_{\gamma}^\beta(e_\alpha)=0$, $\forall \gamma\neq \alpha,
\gamma\neq \beta$. So $D'_{e_\alpha}(\Phi_{e_\beta})(e_\alpha)-
D'_{e_\beta}(\Phi_{e_\alpha})(e_\alpha)=0$ is equivalent to
\[
\omega_{\alpha}^\beta(e_\alpha+e_\beta)=0.\qedhere
\]
\end{proof}

We continue to prove Theorem \ref{existence}. Since $\omega$ is
diagonal, we have $D'(\Phi)=0$ according to Lemma \ref{canoDphi}.
Let us now consider the other equation in \eqref{harmonic}.

\begin{lemma}\label{harm}
Given any quadruple $(M, D^{'}, \Phi, \Phi^{\dag})$, where $M$ is
a complex analytic manifold, $D^{'}$ is a $(1, 0)$ connection on
$\mathcal{T}_M^{(1, 0)}$, $\Phi$ and $\Phi^{\dag}$ are
$\mathcal{C}_M^{\infty}$-linear maps
$$\Phi: \mathcal{T}_M^{1, 0} \rightarrow \mathcal{A}_M^{1, 0} \otimes \mathcal{T}_M^{1, 0},$$
$$\Phi^{\dag}: \mathcal{T}_M^{1, 0} \rightarrow \mathcal{A}_M^{0, 1} \otimes \mathcal{T}_M^{1, 0}.$$
Then the relation $D'\overline{\partial}+\overline{\partial} D' =
-(\Phi \wedge \Phi^{\dagger} + \Phi^{\dagger} \wedge \Phi)$ is
equivalent to
$$\overline{\partial_{z^j}}\omega(\partial_{z^i})= [\tilde{C^{(j)}}, C^{(i)}],\quad \forall i, j.$$
where $z^j$ are any holomorphic local coordinates of $M$, and
$\omega$ is the matrix of connection forms for $D'$. Moreover,
$C^{(i)}$ and $\tilde{C^{(j)}}$ are defined by relations
(\ref{eq:MPhi}) and(\ref{eq:MPhidag}).
\end{lemma}
\begin{proof}
Straightforward computation.
\end{proof}
The quadruple $(M, D^{'}, \Phi, \Phi^{\dag})$ we defined above
satisfies the assumption of Lemma \ref{harm}, so by Lemma
\ref{harm} applied with canonical coordinates, we are reduced to
proving
\begin{equation}\label{overlined}
\overline{\partial_\beta}
\omega(\partial_\alpha)=[\tilde{C^{(\beta)}}, C^{(\alpha)}],
\quad\forall \alpha, \beta,
\end{equation}
where the matrices $C^{(\alpha)}$ and $\tilde{C^{(\beta)}}$ are
defined by (\ref{eq:MPhi}) and (\ref{eq:MPhidag}) in canonical
local coordinates.

Firstly, we will compute the right hand side of (\ref{overlined}).
Obviously, in a system of canonical local coordinates of the
Frobenius manifold, the matrices $C^{(\alpha)}$ satisfy
\begin{equation}\label{eq:canonical}
\begin{cases}
{C^{(\alpha)}}_{\beta}^\gamma=1,& \text{if }\gamma=\alpha=\beta,\\
{C^{(\alpha)}}_{\beta}^\gamma=0,& \text{otherwise}.
\end{cases}
\end{equation}
Now we will compute $\tilde{C^{(\beta)}}$.
\begin{eqnarray*}
\tilde{C^{(\beta)}}
& = & \overline{K} \cdot \overline{C^{(\beta)}} \cdot K \\
& = & \diag\Big(0,\dots, 0,
\frac{|\eta_\beta|}{\overline{\eta_\beta}}
\cdot \frac{|\eta_\beta|}{\eta_\beta},\dots, 0,\dots, 0\Big) \\
& = & \diag(0,\dots, 0, 1, 0,\dots, 0) \\
& = & C^{(\beta)}.
\end{eqnarray*}
Therefore,
\begin{equation*}
[\tilde{C^{(\beta)}}, C^{(\alpha)}]=0,\quad \forall \alpha, \beta.
\end{equation*}
Now we just need to check
$$\overline{e_\beta} \omega(e_\alpha)=0,\quad \forall \alpha, \beta.$$
Since all $\eta_\alpha = g_{\alpha \alpha}$ are nonzero
holomorphic functions, we get that all $\omega_{\alpha}^{\alpha} =
e_\alpha \eta_\alpha/2\eta_\alpha$ are holomorphic. Hence
\begin{equation*}
\overline{e_\beta} \omega_{\gamma}^{\gamma}(e_\alpha) =0,\quad
\forall \alpha, \beta, \gamma.
\end{equation*}
This finishes the proof of the claim. The relation
$\mathcal{Q}^{\dag} = \mathcal{Q}$ is implied by the relation
$\mathcal{Q}:=D'_{\mathcal E}-\mathcal L_{\mathcal
E}-\frac{2-d}2\Id=0$, that we now prove.

We again use a system of canonical local coordinates $u^{\alpha}$.
We can normalize it in such a way that
\begin{equation*}
\mathcal{E}=\sum_\alpha u^\alpha e_\alpha.
\end{equation*}
Therefore,
\begin{equation}\label{LE}
\mathcal{L}_{\mathcal{E}} e_\alpha = - e_\alpha,\quad \forall
\alpha.
\end{equation}
Let us now recall:

\begin{lemma}[\cite{Mani}: Theorem 3.6, p.31]\label{dg}
Let $(M, \circ, g, e,\mathcal{E})$ be a semi-simple Frobenius
manifold and let $u^1, u^2,\dots, u^m$ be a system of canonical
local coordinates of $M$ such that $\mathcal{E}= \sum_{\alpha}
u^\alpha e_\alpha$. Then
\begin{equation*}
\mathcal{L}_{\mathcal{E}}(g)= (2-d)g \Leftrightarrow
\mathcal{L}_{\mathcal{E}} g(\partial_{u^\alpha},
\partial_{u^\alpha})= (-d)\cdot g(\partial_{u^\alpha}, \partial_{u^\alpha}), \quad\forall \alpha.
\end{equation*}
\end{lemma}
Since the relation $\mathcal{L}_{\mathcal{E}}(g)= (2-d)g$ is
included in the definition of a Frobenius manifold, by lemma
\ref{dg}, we have
$$\mathcal{E} \eta_{\alpha}= (-d)\cdot \eta_{\alpha},\quad \forall \alpha,$$
so we have,
\begin{eqnarray*}
D'_{\mathcal{E}} e_\alpha
& = & \omega_{\alpha}^{\alpha}(\mathcal{E})\cdot e_\alpha \\
& = & \frac{\mathcal{E}\eta_{\alpha}}{2 \eta_{\alpha}} \cdot e_\alpha \\
& = & \frac{-d}{2} \cdot e_\alpha
\end{eqnarray*}
Hence, for all $\alpha$, we get
\begin{eqnarray*}
\mathcal{Q} e_\alpha & = & D'_{\mathcal{E}} e_\alpha -
\mathcal{L}_{\mathcal{E}} e_{\alpha} - \frac{2-d}{2} \cdot \Id \\
& = & \frac{-d}{2} \cdot e_\alpha + e_\alpha - \frac{2-d}{2} \cdot \Id \\
& = & 0.
\end{eqnarray*}

By Theorem \ref{suff cond}, we conclude that $(M, g, \circ, e,
\mathcal{E}, \kappa)$ is a CDV$\oplus$-struc\-ture on $M$ with
$\mathcal{Q}=0$.\qed

\begin{remark}
For dimension two, under the assumptions of Theorem
\ref{existence}, we get a CDV$\oplus$-structure on $M$ which is
contained in the discussion of \cite{TA}.
\end{remark}

\subsection{Comparison of three connections and non-K\"ahler property}\label{subsection2c}

In the proof of theorem \ref{flat k}, we will use a system of
canonical local coordinates $u^1, u^2,\dots, u^m$. We normalize
the canonical local coordinates $u^{1}, u^{2},\dots, u^{m}$ in
such a way that $\mathcal{E}= \sum_{\alpha}u^{\alpha}e_{\alpha}$.

\begin{proof}[Proof of theorem \ref{flat k}]
We assume that $\kappa$ is $\nabla$-flat.

\begin{claim*}\label{claim1}
The matrix of connection forms $\omega_{\alpha}^\beta$ for $D^{'}$
is diagonal in $(u^\alpha)$.
\end{claim*}

\begin{proof}
Since $D' = \nabla$, \eqref{harmonic} implies
$$(\Phi
\wedge \Phi^{\dagger} + \Phi^{\dagger} \wedge \Phi)= - (
D'\overline{\partial}+\overline{\partial} D') = - ( \nabla
\overline{\partial}+\overline{\partial} \nabla) =0.$$ By a
straightforward computation, we get
\begin{equation}\label{eq:CC}
[\tilde{C^{(\beta)}}, C^{(\alpha)}]=0,\quad \forall \alpha, \beta.
\end{equation}
Using \eqref{eq:canonical}, \eqref{eq:CC} implies
\begin{equation*}
\overline{K_{\nu \beta}} \cdot K_{\beta \alpha} \cdot
\delta_{\gamma \alpha} = \overline{K_{\alpha \beta}} \cdot
K_{\beta \gamma} \cdot \delta_{\nu \alpha}, \quad\forall \alpha,
\beta, \gamma, \nu.
\end{equation*}
Taking $\nu = \alpha, \gamma \neq \alpha$, we get
\begin{equation}\label{eq:Kabg}
\overline{K_{\alpha \beta}} \cdot K_{\beta \gamma} = 0,
\quad\forall \alpha,\beta,\gamma\text{ with } \alpha \neq \gamma.
\end{equation}
By the non-degeneracy of $\kappa$, for any $\beta$ there exists an
index $\mu_\beta$ such that $K_{{\mu_\beta}\beta}\neq 0$. From
relation \eqref{eq:Kabg}, we have
$$\overline{K_{\mu_{\beta} \beta}} \cdot K_{\beta \gamma} = 0,\quad \forall \beta,\gamma \text{ with }\gamma\neq\mu_\beta,$$
Hence we get
$$K_{\beta \gamma}=0,\quad\forall \gamma\neq\mu_\beta.$$
That is, for any $\beta$, there exists a unique $\mu_\beta$, such
that $K_{\beta{\mu_\beta}}\neq 0$. Similarly, for any $\gamma$,
there exists unique $\nu_\gamma$, such that
$K_{{\nu_\gamma}{\gamma}}\neq 0$.

Assume that there exist $\alpha$ such that $\nu_\alpha \neq
\mu_\alpha$. We will deduce a contradiction. By relation
\eqref{eq:Kabg}, we get
$$\overline{K_{\nu_{\alpha} \alpha}} \cdot K_{\alpha \mu_{\alpha}}=0.$$
Hence we get that either $K_{\nu_{\alpha} \alpha}=0$ or $K_{\alpha
\mu_{\alpha}}=0$. This gives a contradiction.

So we conclude that for any $\alpha$, there exists a unique
$\nu_\alpha$ such that
\begin{equation}\label{eq:Kaba}
\overline{K_{\alpha \nu_{\alpha}}} \cdot K_{\nu_{\alpha} \alpha}
\neq 0.
\end{equation}
Since the relation \eqref{flatk} is included in the definition of
a CDV-structure, by straightforward computations, we have
\begin{equation}\label{eq:omegaK}
\omega  =- \overline{K}\cdot \partial K.
\end{equation}
By relations \eqref{eq:Kabg}, \eqref{eq:Kaba} and
\eqref{eq:omegaK}, we get
\begin{equation}\label{eq:diag omeg}
\omega_{\alpha}^{\beta}= - \delta_{\alpha \beta} \cdot
\overline{K_{\alpha \nu_{\alpha}}} \cdot \partial K_{\nu_{\alpha}
\alpha}, \forall \alpha, \beta.
\end{equation}
Hence we conclude that
\[
\omega = \diag(- \overline{K_{1 \nu_{1}}} \cdot \partial
K_{\nu_{1} 1},\cdots, - \overline{K_{m \nu_{m}}} \cdot \partial
K_{\nu_{m} m}).\qedhere
\]
\end{proof}

Let $\eta$ be the metric potential in the coordinates $u^\alpha$,
i.e.,
\begin{equation*}
\eta_{\alpha} = g_{\alpha \alpha} = e_{\alpha} \eta.
\end{equation*}
We want to prove that $e_{1}, e_{2},\dots, e_{m}$ are
$\nabla$-flat, so we just need to prove all $\eta_{\alpha}$ are
constants, i.e.
\begin{equation*}
\overline{e_{\alpha}} \eta_{\beta} = e_{\alpha}
\eta_{\beta}=0,\quad \forall \alpha, \beta.
\end{equation*}

\begin{claim*}\label{claim2}
$e_{\alpha} \eta_{\beta}=0$, $\forall \alpha, \beta$.
\end{claim*}

\begin{proof}
We have shown that $\omega$ is diagonal, hence
\begin{equation*}
D'_{e_{\alpha}}e_{\beta} = \omega_{\beta}^{\beta}(e_{\alpha})
\cdot e_{\beta}.
\end{equation*}
Let us recall:

\begin{lemma}[\cite{Mani}, proof of Theorem 3.3, p.\ 28-30]\label{last}
Let $(M, g, \circ, e, \mathcal{E})$ be a semi-simple Frobenius
manifold, and let $u^{1}, u^{2},\dots, u^{m}$ be a system of
canonical local coordinates of $M$. Then
\begin{equation*}
\nabla_{\alpha}
e_{\alpha}=\frac{1}{2}\eta_{\alpha}^{-1}e_{\alpha}\eta_{\alpha}\cdot
e_{\alpha}-\sum_{\gamma \neq \alpha}\frac{1}{2}
\eta_{\gamma}^{-1}e_{\gamma}\eta_{\alpha}\cdot e_{\gamma}.
\end{equation*}
For $\alpha \neq \beta$ we have
\begin{equation*}
\nabla_{\alpha}
e_{\beta}=\frac{1}{2}\eta_{\alpha}^{-1}e_{\beta}\eta_{\alpha}\cdot
e_{\alpha}+ \frac{1}{2}
\eta_{\beta}^{-1}e_{\alpha}\eta_{\beta}\cdot e_{\beta}.
\end{equation*}
\end{lemma}

By Lemma \ref{last}, for $\alpha\neq \beta$, we have
\begin{equation*}
\nabla_{\alpha} e_{\beta}=\frac{1}{2} \cdot \eta_{\alpha}^{-1}
\cdot (e_{\beta}\eta_{\alpha})\cdot e_{\alpha}+ \frac{1}{2} \cdot
\eta_{\beta}^{-1} \cdot (e_{\alpha}\eta_{\beta}) \cdot e_{\beta}.
\end{equation*}
From $D' = \nabla$ we get
\begin{equation*}
\omega_{\beta}^{\beta}(e_\alpha) \cdot e_{\beta} = \frac{1}{2}
\cdot \eta_{\alpha}^{-1} \cdot (e_{\beta}\eta_{\alpha})\cdot
e_{\alpha}+ \frac{1}{2} \cdot \eta_{\beta}^{-1} \cdot
(e_{\alpha}\eta_{\beta}) \cdot e_{\beta},\quad \forall \alpha \neq
\beta.
\end{equation*}
Comparing the coefficients of $e_{\beta}$ gives
\begin{equation}\label{ij}
e_{\alpha} \eta_{\beta} =0,\quad \forall \alpha \neq \beta.
\end{equation}

Lemma \ref{last} and (\ref{ij}) give
\begin{equation}\label{eq:aa}
\nabla_{\alpha} e_{\alpha}=\frac{1}{2} \cdot \eta_{\alpha}^{-1}
\cdot (e_{\alpha}\eta_{\alpha})\cdot e_{\alpha}
\end{equation}
and
\begin{equation}\label{eq:ab}
\nabla_{\alpha} e_{\beta} = 0,\quad \forall \alpha \neq \beta.
\end{equation}
Recall also that the unit $e=\sum e_\alpha$ is $\nabla$-flat, i.e.
\begin{equation}\label{eq:ae}
\nabla_{e_{\alpha}} e = 0,\quad \forall \alpha.
\end{equation}
By relations \eqref{eq:aa}, \eqref{eq:ab} and \eqref{eq:ae}, we
get
\begin{equation*}
\eta_{\alpha}^{-1} e_{\alpha} \eta_{\alpha} \cdot e_{\alpha} = 2
\cdot \nabla_{e_{\alpha}} e_{\alpha} = \nabla_{e_{\alpha}} e = 0,
\end{equation*}
i.e.
\begin{equation}\label{ii}
e_{\alpha} \eta_{\alpha} = 0,\quad \forall \alpha.\qedhere
\end{equation}
\end{proof}

From the claim, we get
$$\partial\eta_{\alpha} = 0,\quad \forall \alpha.$$ Since all $\eta_\alpha$ are holomorphic, we get $\overline{\partial} \eta_\alpha=0,\quad \forall
\alpha$, i.e.
\begin{equation}\label{constantg}
d g_{ \alpha \alpha}= d \eta_{\alpha}=0,\quad \forall \alpha.
\end{equation}
Hence we conclude that $u^1, u^2,\dots,u^m$ are $\nabla$-flat
holomorphic local coordinates of $M$. Arguing now as in Theorem
\ref{existence}, we conclude that $\mathcal Q=0$.

Moreover, if $h$ is positive, since $h_{\alpha \alpha}= K_{\alpha
\alpha} \cdot \eta_{\alpha}> 0$, we get
$$\nu_{\alpha}= \alpha,\quad \forall \alpha.$$
So we conclude that $$h_{\alpha \beta } = \delta_{\alpha
\beta}\cdot |\eta_\alpha|.$$

Arguing now as in Theorem \ref{existence}, we conclude that the
matrices of $\kappa$ and $h$ are expressed as in Theorem
\ref{existence} in ($u^{\alpha}$), which ends the proof of Theorem
\ref{flat k}.
\end{proof}

\begin{proof}[Proof of Theorem \ref{non Kahler} \eqref{non symp}]
Recall that we set $h=\hat g-i\hat\omega$ and we have
$\hat\omega=g(\cdot,J\cdot)$. The relation $d \hat \omega \neq 0$ is a
consequence the following lemma.

\begin{lemma}\label{lemma comp}
Let $M$ be a complex analytic manifold. Let $h$ be a Hermitian
pseudo-metric on it. Let $-\hat{\omega}$ be the imaginary part of
$h$. Denote by $D^{'}$ the Chern connection of $h$. Then $$d
\hat{\omega}=0 \Leftrightarrow D'_X Y - D'_Y X = [X, Y],\quad \forall
X, Y \in \mathcal{T}_{M}^{(1, 0)}.$$
\end{lemma}
\begin{proof}
Let $t^1,\dots, t^m$ be any system of holomorphic local
coordinates of $M$. Then
\begin{equation}\label{eq:loc o}
\hat{\omega}= \frac{i}{2} \sum_{k,l}h_{kl} dt^k \wedge
\overline{dt^l},
\end{equation}
hence
\begin{align*}
\partial \hat{\omega}&= \frac{i}{2} \sum_{k,l,p} \partial _{t^p}h_{kl}  dt^p \wedge dt^k \wedge \overline{dt^l}\\
& =\frac{i}{2} \sum_{k< p, l}[\partial _{t^p}h_{kl} -\partial
_{t^k}h_{pl}]dt^p \wedge dt^k \wedge \overline{dt^l}.
\end{align*}
It follows that
\begin{equation}\label{partial omega1}
\partial \hat{\omega}= 0 \Leftrightarrow \partial _{t^k}h_{ij}=
\partial _{t^i}h_{kj},\quad \forall i, j, k.
\end{equation}
Similarly, we get
\begin{equation}\label{partial omega2}
\overline{\partial} \hat{\omega}=0 \Leftrightarrow
\overline{\partial_{t^k}}h_{ji}= \overline{\partial_{t^i}}
h_{jk},\quad \forall i, j, k.
\end{equation}
Since $h$ is Hermitian, this implies $\overline{\partial}
\hat\omega=0 \Leftrightarrow {\partial} \hat\omega=0$, and
therefore
$$d \hat\omega=0 \Leftrightarrow \overline{\partial} \hat\omega=0
\Leftrightarrow {\partial} \hat\omega=0,$$ which is equivalent to
\begin{equation}\label{torsion}
\partial _{t^k}h_{ij}=
\partial _{t^i}h_{kj},\quad \forall i, j, k.
\end{equation}
Denote by $\omega_i^j$ the connection forms of $D'$ with respect
to $(t^i)$. Since $(t^i)$ are holomorphic, we have
\begin{equation}\label{local expr}
\omega_i^j= (\partial h \cdot h^{-1})_{ij}.
\end{equation}
By relation (\ref{local expr}), we conclude that relation
(\ref{torsion}) is equivalent to
$$\omega_{i}^{j}(\partial_{t^k})- \omega_{k}^{j}(\partial_{t^i})=0.$$
However, the above relation is equivalent to
$$D'_{\partial t^i} \partial t^j - D'_{\partial t^j}
\partial t^i -[\partial_{t^i}, \partial_{t^j}] =0,\quad \forall i, j,$$
i.e. $D'$ is torsion free.
\end{proof}
So, by Corollary \ref{have torsion} and Lemma \ref{lemma comp}, we
conclude that $d \hat{\omega} \neq 0$, i.e. $\hat{\omega}$ is not
a symplectic form on $M_{\mathbb{R}}.$
\end{proof}

\begin{proof}[Proof of theorem \ref{non Kahler} \eqref{comparison}]
Assume that $\widetilde{\mathcal{R}e}D^{'}= \hat{{\nabla}}
\mathcal{R}e$, we will derive a contradiction.
\begin{lemma}\label{lemma comp2}
Let $M$ be a complex analytic manifold. Let $h$ be a Hermitian
pseudo-metric on it. Let $\hat{g}$ be the real part of $h$. Denote
by $\hat{\nabla}$ the Levi-Civita connection of $\hat{g}$. If
$\widetilde{\mathcal{R}e}D^{'}= \hat{{\nabla}} \mathcal{R}e$, then
$D^{'}$ is torsion free.
\end{lemma}
\begin{proof}
Since $\widetilde{\mathcal{R}e}D^{'}= \hat{{\nabla}}
\mathcal{R}e$, then for any $X, Y \in \mathcal{T}_M^{(1, 0)}$,
set $X = X_1 + i X_2, Y= Y_1 + i Y_2$, where $X_i, Y_j \in
\mathcal{T}_{M_{\mathbb{R}}}$. By a straightforward computation,
we have
\begin{eqnarray*}
D^{'}_X Y
& = & \mathcal{R}e (D^{'}_X Y)+ i \mathcal{I}m (D^{'}_X Y) \\
& = & \mathcal{R}e (D^{'}_X Y) - i \mathcal{R}e ( D^{'}_X i Y ) \\
& = & \mathcal{R}e (D^{'}_{X_1} Y + i D^{'}_{X_2} Y) - i \mathcal{R}e ( i D^{'}_{X_1} Y + i^2 D^{'}_{X_2} Y) \\
& = & \mathcal{R}e (D^{'}_{X_1} Y +  D^{'}_{X_2} i Y) - i \mathcal{R}e (  D^{'}_{X_1} i Y - D^{'}_{X_2} Y)\\
& = & \mathcal{R}e (D^{'}_{X_1} Y )+ \mathcal{R}e (D^{'}_{X_2} i Y) - i \mathcal{R}e (D^{'}_{X_1} i Y) + i \mathcal{R}e(D^{'}_{X_2} Y)\\
& = & \hat{\nabla}_{X_1} \mathcal{R}e (Y)  + \hat{\nabla}_{X_2} \mathcal{R}e(i Y) - i \hat{\nabla}_{X_1} \mathcal{R}e (i Y)+ i \hat{\nabla}_{X_2} \mathcal{R}e (Y),\\
& = & \hat{\nabla}_{X_1} Y_1 - \hat{\nabla}_{X_2} Y_2 + i \hat{\nabla}_{X_1} Y_2 + i \hat{\nabla}_{X_2} Y_1, \\
& = & (\hat{\nabla}_{X_1} Y_1 - \hat{\nabla}_{X_2} Y_2 ) + i
(\hat{\nabla}_{X_1} Y_2 + \hat{\nabla}_{X_2} Y_1).
\end{eqnarray*}
The sixth equality holds because of
$\widetilde{\mathcal{R}e}D^{'}= \hat{{\nabla}} \mathcal{R}e$ and
$X_1, X_2 \in \mathcal{T}_{M_{\mathbb{R}}}$. Similarly, we get
$$D^{'}_Y X = (\hat{\nabla}_{Y_1} X_1 - \hat{\nabla}_{Y_2} X_2 ) + i
(\hat{\nabla}_{Y_1} X_2 + \hat{\nabla}_{Y_2} X_1).$$ Since
$\hat{\nabla}$ is a torsion-free connection, i.e.
$$\hat{\nabla}_{W_i} W_j - \hat{\nabla}_{W_j} W_i = [W_i, W_j],\quad \forall W_i, W_j \in \mathcal{T}_{M_{\mathbb{R}}},$$
we obtain
\begin{eqnarray*}
D^{'}_X Y - D^{'}_Y X & = & (\hat{\nabla}_{X_1} Y_1
-\hat{\nabla}_{X_2} Y_2 ) + i (\hat{\nabla}_{X_1} Y_2 +
\hat{\nabla}_{X_2} Y_1) \\
& - & (\hat{\nabla}_{Y_1} X_1-
\hat{\nabla}_{Y_2} X_2 ) - i (\hat{\nabla}_{Y_1} X_2 + \hat{\nabla}_{Y_2} X_1) \\
& = & \{[X_1, Y_1] - [X_2, Y_2] \} + i \{[X_1, Y_2] + [X_2, Y_1] \} \\
& = & [X, Y].
\end{eqnarray*}
That is, $D^{'}$ is torsion-free.
\end{proof}
By Corollary \ref{have torsion}, $D'$ is not torsion-free.
However, by Lemma \ref{lemma comp2}, we get that $D^{'}$ is
torsion-free. This gives a contradiction.

In order to end the proof of Theorem \ref{non Kahler}
\eqref{comparison}, it remains to see that $\nabla$ and
$\hat{\nabla}$ cannot coincide. Let us now assume that $t^1,
t^2,\dots, t^m$ are $\nabla$-flat holomorphic local coordinates of
$M$. Set $t^j= x^j +i y^j$, then $x^1, \cdots, x^m, y^1,\dots,
y^m$ is a system of real local coordinates of $M_{\mathbb{R}}$.

\begin{claim*}
If $\widetilde{\mathcal{R}e}\nabla= \hat{{\nabla}} \mathcal{R}e$,
then $x^1,\dots, x^m, y^1,\dots, y^m$ are $\hat{\nabla}$-flat.
\end{claim*}

\begin{proof}
Indeed, for any $X$ in $\mathcal T_M^{(1,0)}$, we then have
$\widetilde{\mathcal{R}e}(\nabla X)= \hat{\nabla}\mathcal{R}e(X)$.
Applying this to $X=\partial_{t^j}$ (resp. $X=i\partial_{t^j}$)
gives the $\hat{\nabla}$-flatness of $\partial_{x^j}$ (resp.
$\partial_{y^j}$).
\end{proof}

So $\hat{g}(\partial_{x^i}, \partial_{x^j})$,
$\hat{g}(\partial_{x^i}, \partial_{y^j})$, and
$\hat{g}(\partial_{y^i},\partial_{y^j})$ are constants, hence all
$\hat{\omega}(\partial_{x^i},\partial_{y^j})=
\hat{g}(\partial_{x^i}, J \partial_{x^j})$ are constant.
Similarly, all $\hat{\omega}(\partial_{x^i},\partial_{x^j})$,
$\hat{\omega}(\partial_{y^i}, \partial_{x^j})$, and
$\hat{\omega}(\partial_{y^i},\partial_{y^j})$ are constant.

By relation \eqref{eq:loc o}, we conclude that all $h_{kl}$ are
constant and from \eqref{partial omega1} and \eqref{partial
omega2} this implies $d \hat{\omega}=0$. We conclude as above to a
contradiction.
\end{proof}

\begin{proof}[Proof of Corollary \ref{connections2}]
Let $t^1, t^2,\dots, t^m$ be a system of $\nabla$-flat holomorphic
local coordinates of $M$ such that $e=\partial_{t_1}$. Denote by
$\omega_{i}^{j}$ the connection forms of $D'$ with respect to
$t^i$. Assume that the matrix $h=(h_{ij})$ is diagonal. We will
derive a contradiction. The matrix $h^{-1}$ is then also diagonal.
It follows from \eqref{local expr} that
$$\omega_{i}^{j}=0, \quad\forall i\neq j.$$
For any CDV-structure, it follows from \eqref{flatk},
\eqref{flath} and Claim \ref{claim0} that the Chern connection
$D'$ and $g$ satisfy
\begin{equation}\label{Dg}
D'(g)=0.
\end{equation}
Under the condition $\omega_{i}^{j}=0$, $\forall i\neq j$,
(\ref{Dg}) is equivalent to
\begin{equation}\label{ome}
(\omega_{i}^{i}+\omega_{j}^{j})g_{ij}=0,\quad \forall i, j.
\end{equation}
By the non-degeneracy of $g$, we know that, for any $j$, there
exists a $k_j$ such that $g_{j k_j}\neq 0$. So from (\ref{ome}),
we deduce
\begin{equation}\label{omega-1}
\omega_j^j= -\omega_{k_j}^{k_j}.
\end{equation}

\begin{claim*}
$\omega_j^j = 0$, $\forall j$.
\end{claim*}

\begin{proof}[Proof of the claim]
For any $j$, set $K_j=\{\ell\mid g_{j\ell}\neq0\}$. If $j\in K_j$,
the claim follows from \eqref{omega-1}. Assume now that $j\not\in
K_j$. The above claim can be deduced from the following lemma.
\begin{lemma}\label{lemma D}
Given any quadruple $(M, \nabla, \Phi, D^{'})$, where $M$ is a
complex analytic manifold, $\nabla$ is a flat holomorphic
connection on $\mathcal{T}_M^{(1, 0)}$, $\Phi$ is a holomorphic
symmetric Higgs field on $M$ such that $\nabla(\Phi)=0$, and
$D^{'}$ is a $(1, 0)$ connection on $\mathcal{T}_M^{(1, 0)}.$
Denote by $t^1, t^2, \cdots, t^m$ a system of holomorphic
$\nabla$-flat local coordinates of $M$, and by $\omega_i^j$ the
connection forms for $D^{'}$. Then
$$D^{'}(\Phi)=0 \Leftrightarrow T_{kj}=T_{jk},\quad \forall j,k \in \{1,
2,\dots, m\},$$ where $T_{kj}:= [C^{(k)}, \omega(j)]$, the
matrices $C^{(j)}$ are defined by the relations \eqref{eq:MPhi}, and $\omega(k)$ are defined by
$$\omega(k) :=(\omega_p^q(\partial_{t^k}))_{m \times m}.$$
\end{lemma}

\begin{proof}[Proof of Lemma \ref{lemma D}]

$D'(\Phi)=0$ is equivalent to
$$D'_{\partial_{t^j}}(\Phi_{\partial_{t^k}})(\partial_{t^i})=D'_{\partial_{t^k}}(\Phi_{\partial_{t^j}})(\partial_{t^i}),\quad\forall i, j, k.$$
Computing
$D'_{\partial_{t^j}}(\Phi_{\partial_{t^k}})(\partial_{t^i})$ we
get
\begin{align*}
D'_{\partial_{t^j}}(\Phi_{\partial_{t^k}})&(\partial_{t^i})
 = D'_{\partial_{t^j}}(\Phi_{\partial_{t^k}}\partial_{t^i})-(\Phi_{\partial_{t^k}})(D'_{\partial_{t^j}} \partial_{t^i}) \\
& = - \sum_{l} D'_{\partial_{t^j}}({C_{ik}}^{l}\partial_{t^l} )- \sum_{l} \omega_i^l (\partial_{t^j}) \Phi_{\partial_{t^k}}(\partial_{t^l} ) \\
& = - \sum_{l} \partial_{t^j}({C_{ik}}^{l}) \partial_{t^l}- \sum_{l}{C_{ik}}^{l} D'_{\partial_{t^j}} \partial_{t^l} - \sum_{l} \omega_i^l (\partial_{t^j}) \Phi_{\partial_{t^k}}(\partial_{t^l} ) \\
& = - \sum_{l} \partial_{t^j}({C_{ik}}^{l}) \partial_{t^l}-
\sum_{l,p}{C_{ik}}^{l} \omega_l^p (\partial_{t^j}) \partial_{t^p}+
\sum_{l,p} \omega_i^l (\partial_{t^j}) {C_{kl}}^p \partial_{t^p}.
\end{align*}
Computing
$D'_{\partial_{t^k}}(\Phi_{\partial_{t^j}})(\partial_{t^i})$
similarly we get
\begin{multline*}
D'_{\partial_{t^k}}(\Phi_{\partial_{t^j}})(\partial_{t^i})\\
= - \sum_{l} \partial_{t^k}({C_{ij}}^{l}) \partial_{t^l}-
\sum_{l,p}{C_{ij}}^{l} \omega_l^p (\partial_{t^k}) \partial_{t^p}+
\sum_{l,p} \omega_i^l (\partial_{t^k}) {C_{jl}}^p \partial_{t^p}.
\end{multline*}
Because of the assumptions $\nabla^2=0$ and $\nabla (\Phi)=0$, we
conclude that
$$- \sum_{l} \partial_{t^j}({C_{ik}}^{l})
\partial_{t^l}= - \sum_{l}
\partial_{t^k}({C_{ij}}^{l}) \partial_{t^l}, \forall i,j, k \in \{1, 2,\dots,
m\}.$$ So $D'(\Phi)=0$ is equivalent to the following relation
\begin{multline*}
- \sum_{l,p}{C_{ik}}^{l} \omega_l^p (\partial_{t^j})
\partial_{t^p}+ \sum_{l,p} \omega_i^l (\partial_{t^j}) {C_{kl}}^p
\partial_{t^p}\\
= - \sum_{l,p}{C_{ij}}^{l} \omega_l^p (\partial_{t^k})
\partial_{t^p}+ \sum_{l,p} \omega_i^l (\partial_{t^k}) {C_{jl}}^p
\partial_{t^p},
\end{multline*}
i.e.,
\begin{equation*}
\sum_{l,p} \Big[ {C_{ik}}^{l} \omega_l^p (\partial_{t^j}) -
\omega_i^l (\partial_{t^j}) {C_{kl}}^p \Big]
\partial_{t^p}= \sum_{l,p} \Big[ {C_{ij}}^{l} \omega_l^p (\partial_{t^k})- \omega_i^l (\partial_{t^k})
{C_{jl}}^p \Big]
\partial_{t^p}.
\end{equation*}
So $D'(\Phi)=0$ is equivalent to
\begin{equation*}
\sum_{l} \Big[{C_{ki}}^{l}\omega^{p}_{l}(\partial_{t^{j}}) -
{C_{kl}}^{p}\omega^{l}_{i}(\partial_{t^{j}})\Big] = \sum_{l}
\Big[{C_{ji}}^{l}\omega^{p}_{l}(\partial_{t^{k}}) -
{C_{jl}}^{p}\omega^{l}_{i}(\partial_{t^{k}})\Big],
\end{equation*}
$\forall i,j,k,p \in \{1, 2,\dots, m\}$, which is the desired
relation.
\end{proof}
Let us continue to prove the claim. Since for any CDV-structure,
the underlying quadruple $(M, \nabla, \Phi, D^{'})$ satisfies the
assumption of lemma \ref{lemma D}, and the relation
$D^{'}(\Phi)=0$ is included in the definition of a CDV-structure,
we get
\begin{equation}\label{tkjtjk}
[C^{(j)}, \omega(k)]=[C^{(k)}, \omega(j)],\quad \forall k, j.
\end{equation}
Note also that ${C_{j1}}^q=\delta_{jq}$ because
$\partial_{t_1}=e$. So for any $p, q$, from relation
\eqref{tkjtjk}, we have
$$\sum_{l}[{C_{jp}}^l \omega_l^q(\partial_{t^k}) -
\omega_p^l(\partial_{t^k}){C_{jl}}^q] = \sum_{l}[{C_{kp}}^l
\omega_l^q(\partial_{t^j}) -
\omega_p^l(\partial_{t^j}){C_{kl}}^q].
$$
Taking $p=1, q=j$, we get
\begin{equation}\label{omega0}
\omega_j^j(\partial_{t^k})=\omega_1^1(\partial_{t^k}),\quad
\forall k\neq j.
\end{equation}
For any $k$, because $\dim_{\mathbb{C}}M \geq 3$, there exists
$j\neq k$ such that there exists $k_j\in K_j$ with $k_j\neq k$
(otherwise, there would exist $k$ such that, for any $j\neq k$, we
have $K_j=\{k\}$; this would imply that the matrix $(g_{j\ell})$
has zero entries except in the $k$th line and the $k$th column, so
its rank is at most two, which contradicts its invertibility when
$\dim M\geq3$). Then, by (\ref{omega-1}) and (\ref{omega0}),
$$\omega_1^1(\partial_{t^k})=\omega_j^j(\partial_{t^k}) = -
\omega_{k_j}^{k_j}(\partial_{t^k})=-\omega_1^1(\partial_{t^k}),\quad
\forall k,$$ hence $\omega_1^1=0$ and
$\omega_j^j(\partial_{t^k})=0$ $\forall k\neq j$. On the other
hand, if $k_j\in K_j$, we have $j\neq k_j$, so
\begin{equation*}
\omega_j^j(\partial_{t^j})= -\omega_{k_j}^{k_j}(\partial_{t^j}) =
-\omega_1^1(\partial_{t^j}) =0.\qedhere
\end{equation*}
\end{proof}

The claim implies
$$D'\partial_{t^i}=\sum_{j}\omega_i^j
\partial_{j} = 0,\quad \forall i.$$
Hence we have $$\nabla \partial_{t^j} = 0 = D' \partial_{t^j},
\quad\forall j,$$ that is, $$D'= \nabla.$$ However, by Corollary
\ref{D nabla}, we know that $D'\neq \nabla$. This gives a
contradiction.
\end{proof}

\section{Applications}

\subsection{Some consequences of Theorem \ref{suff cond}}\label{subsection2b}

Given any Frobenius manifold $M$ of dimension $m$, if we suppose
that $g(e, e)=0$ and that the eigenvalues of $\nabla \mathcal{E}$
are simple, then by Proposition \ref{proposition4}, we can choose
$\nabla$-flat holomorphic local coordinates $t^1, t^2,\dots, t^m$
of $M$ such that the relations (\ref{gij}), (\ref{et1}),
(\ref{LieF}), (\ref{E}), (\ref{d1}), (\ref{did}) and (\ref{dFd})
hold. We will give a local expression of the relations given in
the definition of CDV-structure in these local coordinates. Let us
denote by $\omega$ the matrix $(\omega^{l}_{k})$ defined by
\begin{equation}\label{eq:omega}
D'\partial_{t_k}=\sum_l\omega^{l}_{k}\partial_{t_l}
\end{equation}
in any system of holomorphic local coordinates $t^1, t^2, \dots,
t^m$ of $M$.

\begin{lemma}\label{lemma C}
Under the assumptions of Theorem \ref{suff cond} and Proposition
\ref{proposition4}, the relation $D'(g)=0$ is equivalent to
$\omega^{j}_{i}+\omega^{m+1-i}_{m+1-j} =0$, $\forall i,j \in \{1,
2,\dots, m\}$.
\end{lemma}
\begin{proof}
By definition,
\begin{equation*}
D_h(g)=0.
\end{equation*}
is equivalent to
\begin{equation*}
0=dg_{ij}=g(D_h \partial_{t^{i}}, \partial_{t^{j}})+g(
\partial_{t^{i}}, D_h \partial_{t^{j}}),
\end{equation*}
that is, to
\begin{equation*}
\sum_{k}\omega^{k}_{i} g_{kj}+\sum_{k}\omega^{k}_{j} g_{ki} =0.
\end{equation*}
Now $g_{ij}= \delta_{i+j, m+1}$ holds, so this is equivalent to
\begin{equation*}
\omega^{m+1-j}_{i}+\omega^{m+1-i}_{j} =0, \quad\forall i,j,
\end{equation*}
and changing notation, this is equivalent to
\begin{equation*}
\omega^{j}_{i}+\omega^{m+1-i}_{m+1-j} =0, \quad\forall i,j \in
\{1, 2,\dots, m\}.\qedhere
\end{equation*}
\end{proof}

\begin{remark}
Under the assumptions of Theorem \ref{suff cond} and Proposition
\ref{proposition4}, we can get the relation $D'(g)=0$ in another
way. In fact by Lemma \ref{lemma A}, we have
$$h^{-1}= g \cdot
\overline{h} \cdot g,$$ so
\begin{equation*}
\sum_{k} h_{ik} \cdot \overline{h}_{m+1-k,m+1-j}= \delta_{ij}.
\end{equation*}
So we get
\begin{equation*}
\sum_{k} \partial h_{ik} \cdot \overline{h}_{m+1-k,m+1-j}+\sum_{k}
{h}_{ik} \cdot \partial \overline{h}_{m+1-k,m+1-j} =0.
\end{equation*}

Now we compute $\omega_{i}^{j}$ first. we get $\forall i, j \in \{
1, 2,\dots, m\}$,
\begin{eqnarray*}
\omega_{i}^{j}
& = & \sum_k \partial h_{ik} \cdot h^{kj}\\
& = & \sum_{k,l,p} \partial h_{ik} \cdot g_{kl}\cdot \overline{h}_{lp} \cdot g_{pj} \\
& = & \sum_{k,l} \partial h_{ik} \cdot g_{kl} \cdot \overline{h}_{l,m+1-j} \\
& = & \sum_{k} \partial h_{ik} \cdot \overline{h}_{m+1-k,m+1-j},
\end{eqnarray*}
hence we have
\begin{eqnarray*}
\omega_{m+1-j}^{m+1-i} & = & \sum_{k} \partial h_{m+1-j,k} \cdot
\overline{h}_{m+1-k,i}\\
& = & \sum_{k} \partial h_{m+1-j,m+1-k} \cdot
\overline{h}_{ki}\\
& = & \sum_{k} \partial \overline{h}_{m+1-k,m+1-j} \cdot {h}_{ik}
\end{eqnarray*}
i.e.,
\begin{equation*}
\omega_{i}^{j} + \omega_{m+1-j}^{m+1-i}=0.
\end{equation*}
\end{remark}

\begin{remark}
Under the assumptions of Theorem \ref{suff cond} and Proposition
\ref{proposition4}, we get
\begin{enumerate}
\item $m=2$ implies that $\omega^{1}_{2}=\omega^{2}_{1}=0$,
$\omega^{2}_{2}=-\omega^{1}_{1}$, so if we want to compute all
$\omega^{i}_{j}$, we just need compute $\omega^{1}_{1}$.

\item $m=3$ implies that
$\omega^{1}_{3}=\omega^{3}_{1}=\omega^{2}_{2}=0$,
$\omega^{3}_{2}=-\omega^{2}_{1}$,
$\omega^{2}_{3}=-\omega^{1}_{2}$,
$\omega^{3}_{3}=-\omega^{1}_{1}$, so if we want to compute all
$\omega^{i}_{j}$, we just need compute $\omega^{1}_{1}$,
$\omega^{2}_{1}$, $\omega^{1}_{2}$.
\end{enumerate}
\end{remark}

\begin{lemma}\label{lemma E}
Under the assumptions of Theorem \ref{suff cond} and Proposition
\ref{proposition4}, if $m=3$, then the relation $D'(\Phi)=0$ is
equivalent to the following three relations
$$
\left\{
\begin{array}{lll}
\omega^{2}_{1}(\partial_{t^{3}})=\omega^{1}_{1}(\partial_{t^{2}}),
\\[4mm]
\omega^{1}_{1}(\partial_{t^{3}}) =
C_{223}\omega^{2}_{1}(\partial_{t^{2}}) +
\omega^{1}_{2}(\partial_{t^{2}})-
C_{222}\omega^{1}_{1}(\partial_{t^{2}}),
\\[4mm]
\omega^{1}_{2}(\partial_{t^{3}})=C_{223}\omega^{1}_{1}(\partial_{t^{2}})
- C_{233}\omega^{2}_{1}(\partial_{t^{2}}).
\end{array}
\right.
$$
\end{lemma}

\begin{proof}
Because $\omega(e)=0$, we just need to compute
$\omega(\partial_{t^2})$ and $\omega(\partial_{t^3})$. From Lemma
\ref{lemma D}, we know that $D'(\Phi)=0$ is equivalent to
$T_{jk}=T_{kj}$ $\forall j,k \in \{1, 2,\dots, m\}$.

Because when $j=k$, the above relations hold automatically, so the
non-trivial cases are that $j=2, k=3$ or $j=3,k=2$. However $j=2,
k=3$ and $j=3,k=2$ give the same relations.

$Step$ $1^{\circ}$ Take $p=3$. We have
$$T_{kj}=C_{2ik}\omega_{2}^{3}(\partial_{t^{j}})+
g_{ik}\omega_{3}^{3}(\partial_{t^{j}})-
\omega_{i}^{4-k}(\partial_{t^{j}})$$

By the relation $\omega_{i}^{j}+\omega_{3+1-j}^{3+1-i}=0$ we have
\begin{equation*}
T_{kj}= -C_{2ik}\omega_{1}^{2}(\partial_{t^{j}})-
g_{ik}\omega_{1}^{1}(\partial_{t^{j}})-
\omega_{i}^{4-k}(\partial_{t^{j}})
\end{equation*}

If $i=1$, then we get
\begin{equation*}
T_{kj}^{i=1}= -g_{k2}\omega_{1}^{2}(\partial_{t^j}) - g_{k1}
\omega_{1}^{1}(\partial_{t^j}) - \omega_{1}^{4-k}(\partial_{t^j}).
\end{equation*}
Then $T_{23}=T_{32}$ is equivalent to
\begin{align}\label{31}
\omega^{2}_{1}(\partial_{t^{3}})=\omega^{1}_{1}(\partial_{t^{2}}).
\end{align}

If $i=2$, then $T_{23}=T_{32}$ is equivalent to
\begin{equation*}
C_{222}\omega^{2}_{1}(\partial_{t^{3}})+\omega^{1}_{1}(\partial_{t^{3}})+
\omega^{2}_{2}(\partial_{t^{3}})=C_{223}\omega^{2}_{1}(\partial_{t^{2}})+0+\omega^{1}_{2}(\partial_{t^{2}}).
\end{equation*}
$\omega^{2}_{1}(\partial_{t^{3}})=\omega^{1}_{1}(\partial_{t^{2}})$
and $\omega_{i}^{m+1-i}=0$ hold, so
\begin{align}\label{32}
C_{222}\omega^{1}_{1}(\partial_{t^{2}})+\omega^{1}_{1}(\partial_{t^{3}})
- C_{223}\omega^{2}_{1}(\partial_{t^{2}})-
\omega^{1}_{2}(\partial_{t^{2}})=0.
\end{align}

If $i=3$, then $T_{23}=T_{32}$ is equivalent to
\begin{align}\label{33}
C_{223}\omega^{1}_{1}(\partial_{t^{2}})-
C_{233}\omega^{2}_{1}(\partial_{t^{2}})-\omega^{1}_{2}(\partial_{t^{3}})=0.
\end{align}

$Step$ $2^{\circ}$ Take $p=2$. we have
$$T_{kj}=C_{3ik}\omega_{1}^{2}(\partial_{t^{j}})-
g_{ik}\omega_{2}^{1}(\partial_{t^{j}})- \sum_{l} C_{2kl}
\omega_{i}^{l}(\partial_{t^{j}})$$

If $i=1$, then $T_{23}=T_{32}$ is equivalent to
\begin{align}\label{21}
C_{222}\omega^{1}_{1}(\partial_{t^{2}})+\omega^{1}_{1}(\partial_{t^{3}})-C_{223}\omega^{2}_{1}(\partial_{t^{2}})-\omega^{1}_{2}(\partial_{t^{2}})=0.
\end{align}

If $i=2$, then $T_{23}=T_{32}$ is equivalent to
\begin{align}\label{22}
C_{223}\omega^{1}_{1}(\partial_{t^{2}}) -
C_{233}\omega^{2}_{1}(\partial_{t^{2}})-
\omega^{1}_{2}(\partial_{t^{3}})=0.
\end{align}

If $i=3$, then $T_{23}=T_{32}$ is equivalent to
\begin{align}\label{23}
C_{223}\omega^{1}_{1}(\partial_{t^{3}})-
C_{333}\omega^{2}_{1}(\partial_{t^{2}})-C_{223}\omega^{1}_{2}(\partial_{t^{2}})+C_{222}\omega^{1}_{2}(\partial_{t^{3}})=0.
\end{align}

$Step$ $3^{\circ}$ Take $p=1$. we have
$$T_{kj}=C_{3ik}\omega_{1}^{1}(\partial_{t^{j}})+
C_{2ik}\omega_{2}^{1}(\partial_{t^{j}})- \sum_{l} C_{3kl}
\omega_{i}^{l}(\partial_{t^{j}}).$$

If $i=1$, then computing it directly we have
\begin{align}\label{11}
C_{223}\omega^{1}_{1}(\partial_{t^{2}}) -
C_{233}\omega^{2}_{1}(\partial_{t^{2}})-
\omega^{1}_{2}(\partial_{t^{3}})=0.
\end{align}

If $i=2$, then we have
\begin{align}\label{12}
C_{223}\omega^{1}_{1}(\partial_{t^{3}})- C_{333}
\omega^{2}_{1}(\partial_{t^{2}})-C_{223}\omega^{1}_{2}(\partial_{t^{2}})+C_{222}\omega^{1}_{2}(\partial_{t^{3}})=0.
\end{align}

If $i=3$, then we have
\begin{align}\label{13}
C_{333}\omega^{1}_{1}(\partial_{t^{2}})-
C_{233}\omega^{1}_{1}(\partial_{t^{3}})+C_{233}\omega^{1}_{2}(\partial_{t^{2}})-C_{223}\omega^{1}_{2}(\partial_{t^{3}})=0.
\end{align}

So we just need five relations (\ref{31}), (\ref{32}), (\ref{33}),
(\ref{23}) and (\ref{13}), i.e.,
\begin{equation*}
\omega^{2}_{1}(\partial_{t^{3}})=\omega^{1}_{1}(\partial_{t^{2}}),
\end{equation*}
\begin{equation*}
C_{222}\omega^{1}_{1}(\partial_{t^{2}}) +
\omega^{1}_{1}(\partial_{t^{3}}) -
C_{223}\omega^{2}_{1}(\partial_{t^{2}}) -
\omega^{1}_{2}(\partial_{t^{2}})=0,
\end{equation*}
\begin{equation*}
C_{223}\omega^{1}_{1}(\partial_{t^{2}}) -
C_{233}\omega^{2}_{1}(\partial_{t^{2}}) -
\omega^{1}_{2}(\partial_{t^{3}})=0,
\end{equation*}
\begin{equation*}
C_{223}\omega^{1}_{1}(\partial_{t^{3}}) -C_{333}
\omega^{2}_{1}(\partial_{t^{2}}) -
C_{223}\omega^{1}_{2}(\partial_{t^{2}}) +
C_{222}\omega^{1}_{2}(\partial_{t^{3}})=0,
\end{equation*}
\begin{equation*}
C_{333}\omega^{1}_{1}(\partial_{t^{2}}) -C_{233}
\omega^{1}_{1}(\partial_{t^{3}}) +
C_{233}\omega^{1}_{2}(\partial_{t^{2}}) -
C_{223}\omega^{1}_{2}(\partial_{t^{3}})=0,
\end{equation*}

\begin{claim*}
The relations (\ref{32}) and (\ref{33}) together with
WDVV-equation imply the relations (\ref{23}) and (\ref{13}).
\end{claim*}

In fact, $m=3$, so WDVV-equations is just
\begin{align}
(C_{223})^2 - C_{222} \cdot C_{233}- C_{333}= 0.\label{WDVV}
\end{align}
From the relation (\ref{32}) we get
\begin{align}
\omega^{1}_{1}(\partial_{t^{3}})= -
C_{222}\omega^{1}_{1}(\partial_{t^{2}}) +
C_{223}\omega^{2}_{1}(\partial_{t^{2}}) +
\omega^{1}_{2}(\partial_{t^{2}}).\label{omega113}
\end{align}

From the relation (\ref{33}) we get
\begin{align}
\omega^{1}_{2}(\partial_{t^{3}}) =
C_{223}\omega^{1}_{1}(\partial_{t^{2}}) -
C_{233}\omega^{2}_{1}(\partial_{t^{2}}).\label{omega213}
\end{align}

So computing directly we prove that relations (\ref{WDVV}),
(\ref{omega113}) and (\ref{omega213}) imply relations (\ref{23})
and (\ref{13}).
\end{proof}

\begin{remark}
Under the assumptions of Theorem \ref{suff cond} and Proposition
\ref{proposition4}, if $m=2$, then the relation $D'(\Phi)=0$ is
equivalent to
$$D'_e \partial_{t^2}= \partial_{t^2} \circ
D'_e e.$$

From $D'(\Phi)=0$ and $D'_{e}e=0$, we have $\omega^{1}_{1}(e)=0$,
i.e., $eh_{ij}=0$, $\forall i,j \in \{1, 2\}$. So computing
$\omega_1^1$, we just need to compute
$\omega_1^1(\partial_{t^2})$.
\end{remark}

\begin{lemma}\label{lemma 2B}
If $m=2$, under the assumption of Theorem \ref{suff cond} and
Proposition \ref{proposition4}, we have $h(a, b)=g(a, \kappa b)$,
$\kappa^2 = \Id$, which are equivalent to the following relations:
$$
\left\{
\begin{array}{lll}
|h_{12}|^2+h_{11}h_{22}=1,
\\[4mm]
h_{11}h_{12}=0,
\\[4mm]
h_{22}h_{12}=0.
\end{array}
\right.
$$
\end{lemma}
\begin{proof}
Straightforward computation.
\end{proof}

\begin{remark}
If $h$ is positive, then $h_{ii} > 0$, so $h$ is diagonal
$$h= \diag(h_{11},
h_{11}^{-1}).$$ Otherwise if $h_{11}=0$ or $h_{22}=0$, we get
$$h_{11}=h_{22}=0\quad \text{and}\quad\vert h_{12} \vert =1.$$
\end{remark}

\begin{lemma}\label{lemma B}
If $m=3$, under the assumption of Theorem \ref{suff cond} and
Proposition \ref{proposition4}, we have $h(a, b)=g(a, \kappa b)$,
$\kappa^2 = \Id$, which are equivalent to the following relations:
$$
\left\{
\begin{array}{lll}
h_{11}h_{33}+h_{12}h_{32}+|h_{13}|^{2}=1,
\\[4mm]
2h_{21}h_{23}+(h_{22})^{2}=1,
\\[4mm]
h_{11}h_{23}+h_{12}h_{22}+h_{13}h_{21}=0,
\\[4mm]
2h_{11}h_{13}+(h_{12})^{2}=0,
\\[4mm]
h_{12}h_{33}+h_{22}h_{23}+h_{13}h_{32}=0,
\\[4mm]
2h_{13}h_{33}+(h_{23})^{2}=0.
\end{array}
\right.
$$
\end{lemma}
\begin{proof}
Straightforward computation.
\end{proof}

\begin{lemma}\label{lemma I}
Under the assumptions of Theorem \ref{suff cond} and Proposition
\ref{proposition4}, $\mathcal{Q}=\mathcal{Q}^{\dag}$ is equivalent
to the following relation
\begin{equation*}
(\mathcal{E}- \overline{\mathcal{E}}) h = h \cdot A- A \cdot h,
\end{equation*}
where the matrix $ A_{ij}= \delta_{ij} d_{i}$. That is
\begin{equation*}
(\mathcal{E}- \overline{\mathcal{E}}) h_{ij}=(d_j - d_i) \cdot
h_{ij}, \forall i,j.
\end{equation*}
\end{lemma}
\begin{proof}
$\mathcal{Q}= D'_{\mathcal{E}}- \mathcal{L}_{\mathcal{E}} -
\frac{2-d}{2}\cdot \Id$, so
\begin{eqnarray*}
h(\mathcal{Q}(\partial_{t^{i}}), \partial_{t^{j}})
& = & h(D'_{\mathcal{E}} \partial_{t^{i}}, \partial_{t^{j}} )-h(\mathcal{L}_{\mathcal{E}} \partial_{t^{i}}, \partial_{t^{j}})- \frac{2-d}{2}h_{ij} \\
& = & (\omega(\mathcal{E}) \cdot h)_{ij} + d_{i} \cdot h_{ij} -
\frac{2-d}{2}h_{ij}
\end{eqnarray*}
similarly, we have
\begin{eqnarray*}
h(\partial_{t^{i}}, \mathcal{Q}(\partial_{t^{j}}))
& = & h(\partial_{t^{i}},D'_{\mathcal{E}}(\partial_{t^{j}})-h(\partial_{t^{i}},\mathcal{L}_{\mathcal{E}}\partial_{t^{j}})-\frac{2-d}{2}h_{ij} \\
& = & \overline{(\omega(\mathcal{E}) \cdot h)_{ji}} + d_{j} \cdot
h_{ij} - \frac{2-d}{2}h_{ij}
\end{eqnarray*}
$\mathcal{Q}=\mathcal{Q}^{\dag}$ means that
\begin{equation*}
(\omega(\mathcal{E}) \cdot h)_{ij} + d_{i} \cdot h_{ij} =
\overline{(\omega(\mathcal{E}) \cdot h)_{ji}} + d_{j} \cdot h_{ij}
\end{equation*}
We will simplify $\omega(\mathcal{E}) \cdot h$ and
$\overline{\omega(\mathcal{E}) \cdot h}$ in the following
\begin{equation*}
\omega(\mathcal{E}) \cdot h = \mathcal{E}(h) \cdot h^{-1} \cdot h
= \mathcal{E}(h).
\end{equation*}
So
\begin{equation*}
\overline{\omega(\mathcal{E}) \cdot h} = \overline{\mathcal{E}(h)}
= \overline{\mathcal{E}}(\overline{h})=
\overline{\mathcal{E}}(h^{t}).
\end{equation*}
Similarly, we have $\overline{(\omega(\mathcal{E}) \cdot h)_{ji}}
= \overline{\mathcal{E}}(h_{ij})$.

So $\mathcal{Q}=\mathcal{Q}^{\dag}$ is equivalent to
\begin{equation*}
(\mathcal{E}(h)- \overline{\mathcal{E}}(h))_{ij}= (d_j -d_i) \cdot
h_{ij} = (h \cdot A - A \cdot h)_{ij}, \quad\forall i,j.\qedhere
\end{equation*}
\end{proof}

\begin{remark}
For $m=2$, under the assumptions of Theorem \ref{suff cond} and
Proposition \ref{proposition4}, if $h$ is positive definite, then
$h = \diag(h_{11}, h_{11}^{-1})$, hence
$\mathcal{Q}=\mathcal{Q}^{\dag}$ is equivalent to the following
relation
\begin{equation*}
\mathcal{E}h_{11}=\overline{\mathcal{E}}h_{11}.
\end{equation*}
\end{remark}

From above discussion, we know that given any Frobenius manifold
$(M, g, \circ, e, \mathcal{E})$ of dimension two and an
anti-linear involution $\kappa$ of $\mathcal{T}_M^{(1,0)}$, they
defined a positive CDV-structure if and only if the following
relations hold
$$
\left\{
\begin{array}{lll}
h= \diag(h_{11}, h_{11}^{-1}),
\\[4mm]
h_{11}=h_{11}(t_2),
\\[4mm]
\overline{\partial_{t^2}} \partial_{t^2} \log h_{11}=h_{11}^2
{|\partial_2^3 F|}^2 - h_{11}^{-2},
\\[4mm]
\mathcal{E}h_{11}=\overline{\mathcal{E}}h_{11}.
\end{array}
\right.
$$

The integrability of these equations is proved by Atsushi
Takahashi in \cite{TA}, here ``integrability" means that there
exist a real analytic function $h_{11}$ on $M$ such that it
satisfies the above equations. For $m \geq 3$, the sufficient and
necessary conditions of CDV-structures simplify the equations in
the definition of the structure, but it is still not easy to see
the integrability of the determined equations of $h$.

\subsection{Application to harmonic Frobenius manifolds}\label{subsection2e}
In this section, we exhibit a harmonic potential $P$ (in the sense
of \cite{Sabb22}) for the CDV$\oplus$ structures given by Theorem
\ref{existence}, from which we keep the notations. Recall that $P$
is an endomorphism of $\mathcal{T}_M^{(1, 0)}$ which is in
particular a solution to
\[
\begin{cases}
D'P= \Phi,\\
D'=\nabla - [P^{\dag}, \Phi].
\end{cases}
\]

We denote by $(P_\alpha^\beta)$ the matrix of $P$ in the fixed
system of canonical local coordinates, i.e.
\begin{equation*}
P e_{\alpha} = \sum_\beta P_{\alpha}^\beta e_\beta.
\end{equation*}
From Theorem \ref{existence} and Lemma \ref{last}, we deduce that
the relations
\begin{equation*}
D'_{e_\alpha}e_\beta=\nabla_{e_\alpha}e_\beta -
[P^{\dag},\Phi_{e_\alpha}](e_\beta).\quad \forall\alpha, \beta
\end{equation*}
are equivalent to
\begin{equation}\label{eq:poten}
P_{\beta}^{\alpha} = \frac{\overline{\eta_{\alpha \beta}}}{2
|\eta_\alpha \eta_\beta|}\cdot \eta_\beta,\quad \forall \alpha\neq
\beta,
\end{equation}
where $\eta_{\alpha \beta}:= e_\alpha \eta_\beta = e_\alpha
e_\beta \eta$.

By a similar computation,
\begin{equation*}
D'_{e_\alpha}(P)(e_\beta)= \Phi_{e_\alpha}e_\beta,\quad\forall
\alpha,\beta
\end{equation*}
are equivalent to
\[
\begin{cases}
e_\alpha(P_{\beta}^\beta)= - \delta_{\alpha \beta},\\
e_\alpha(P_\beta^\gamma)= P_\beta^\gamma \cdot
[\omega_\beta^\beta(e_\alpha)-\omega_\gamma^\gamma(e_\alpha)],&
\forall \alpha\neq \beta, \forall \gamma,\\
e_\alpha(P_{\alpha}^\beta)= P_{\alpha}^\beta \cdot
[\omega_{\alpha}^{\alpha}(e_\alpha)-\omega_\beta^\beta(e_\alpha)],&
\forall \alpha\neq \beta.
\end{cases}
\]

So if we take the matrix of $P$ as follows
\begin{equation}\label{eq:potential}
\begin{cases}
P_\beta^\beta= - u^\beta,\\[5pt]
P_\beta^\alpha = \dfrac{\overline{\eta_{\alpha \beta}}}{2
|\eta_\alpha \eta_\beta|}\cdot \eta_\beta,& \forall \alpha\neq
\beta.
\end{cases}
\end{equation}
then the endomorphism $P$ satisfies $D'P= \Phi$ and $D'=\nabla -
[P^{\dag}, \Phi]$. It is easy to see this $P$ also satisfies
\begin{equation*}
P^* = P,
\end{equation*}
Set $\mathcal{V}:= \nabla \mathcal{E} - \frac{2-d}{2} \cdot \Id$.
We will deduce the relation $\mathcal{Q}=\mathcal{V}+[P^{\dag},
\mathcal{U}]$ in the following.

Since $${P^{\dag}}_{\alpha}^{\beta}= \overline{K_{\alpha \alpha}}
\cdot \overline{P_\alpha^\beta} \cdot K_{\beta \beta},$$ computing
the right hand side directly, we get
\begin{equation}\label{eq:pdag}
\begin{cases}
{P^\dag}_\beta^\beta= - \overline{u^\beta},\\[5pt]
{P^\dag}_\alpha^\beta = \omega_{\beta}^{\beta}(e_{\alpha}),
\forall \alpha\neq \beta.
\end{cases}
\end{equation}
By straightforward computations, we get
\begin{equation}\label{eq:pu}
[P^\dag, \mathcal{U}] e_\alpha = -\sum_{\beta \neq \alpha}
(u^\beta - u^\alpha) \cdot
 \omega_{\beta}^{\beta}(e_\alpha)\cdot e_\beta = -\sum_{\alpha \neq \beta} (u^\beta
- u^\alpha)\cdot \frac{e_\alpha \eta_\beta}{2 \eta_\beta}\cdot
e_\beta.
\end{equation}

By a similar computation, we get $$\nabla_{e_\alpha} \mathcal{E} =
\frac{2-d}{2}\cdot e_\alpha + \sum_{\beta \neq \alpha} (u^\beta -
u^\alpha)\cdot \frac{e_\alpha \eta_\beta}{2 \eta_\beta} \cdot
e_\beta.$$ Hence we get
\begin{equation}\label{eq:v}
\mathcal{V} e_\alpha =   \sum_{\beta \neq \alpha} (u^\beta -
u^\alpha) \cdot \frac{e_\alpha \eta_\beta}{2 \eta_\beta}\cdot
e_\beta.
\end{equation}
By relations \eqref{eq:pu} and \eqref{eq:v}, we conclude that
$$\mathcal{V} + [P^\dag,
\mathcal{U}]= 0 = \mathcal{Q}.$$

So we have proved

\begin{proposition}\label{potential}
The CDV$\oplus$-structure constructed from Theorem \ref{existence}
is a harmonic Frobenius manifold with the harmonic potential $P$
given by \eqref{eq:potential}.
\end{proposition}

\bibliographystyle{amsplain}
\bibliography{lin090305_cs}

\providecommand{\bysame}{\leavevmode\hbox to3em{\hrulefill}\thinspace}
\providecommand{\MR}{\relax\ifhmode\unskip\space\fi MR }
% \MRhref is called by the amsart/book/proc definition of \MR.
\providecommand{\MRhref}[2]{%
  \href{http://www.ams.org/mathscinet-getitem?mr=#1}{#2}
}
\providecommand{\href}[2]{#2}
\begin{thebibliography}{10}

\bibitem{CV3}
Sergio Cecotti and Cumrun Vafa, \emph{Topological--anti-topological fusion},
  Nuclear Phys. B \textbf{367} (1991), no.~2, 359--461. \MR{MR1139739
  (93a:81168)}

\bibitem{CV2}
\bysame, \emph{Exact results for supersymmetric {$\sigma$} models}, Phys. Rev.
  Lett. \textbf{68} (1992), no.~7, 903--906. \MR{MR1148809 (93b:81281)}

\bibitem{CV1}
\bysame, \emph{Massive orbifolds}, Modern Phys. Lett. A \textbf{7} (1992),
  no.~19, 1715--1723. \MR{MR1168619 (93d:81138)}

\bibitem{CVN}
\bysame, \emph{On classification of {$N=2$} supersymmetric theories}, Comm.
  Math. Phys. \textbf{158} (1993), no.~3, 569--644. \MR{MR1255428 (95g:81198)}

\bibitem{Dubro}
Boris Dubrovin, \emph{Geometry and integrability of topological-antitopological
  fusion}, Comm. Math. Phys. \textbf{152} (1993), no.~3, 539--564.
  \MR{MR1213301 (95a:81227)}

\bibitem{D}
\bysame, \emph{Geometry of {$2$}{D} topological field theories}, Integrable
  systems and quantum groups (Montecatini Terme, 1993), Lecture Notes in Math.,
  vol. 1620, Springer, Berlin, 1996, pp.~120--348. \MR{MR1397274 (97d:58038)}

\bibitem{Hert}
Claus Hertling, \emph{Frobenius manifolds and moduli spaces for singularities},
  Cambridge Tracts in Mathematics, vol. 151, Cambridge University Press,
  Cambridge, 2002. \MR{MR1924259 (2004a:32043)}

\bibitem{Hert2}
\bysame, \emph{{$tt\sp *$} geometry, {F}robenius manifolds, their connections,
  and the construction for singularities}, J. Reine Angew. Math. \textbf{555}
  (2003), 77--161. \MR{MR1956595 (2005f:32049)}

\bibitem{Mani}
Yuri~I. Manin, \emph{Frobenius manifolds, quantum cohomology, and moduli
  spaces}, American Mathematical Society Colloquium Publications, vol.~47,
  American Mathematical Society, Providence, RI, 1999. \MR{MR1702284
  (2001g:53156)}

\bibitem{Sabb}
Claude Sabbah, \emph{D\'eformations isomonodromiques et vari\'et\'es de
  {F}robenius}, Savoirs Actuels (Les Ulis). [Current Scholarship (Les Ulis)],
  EDP Sciences, Les Ulis, 2002, , Math\'ematiques (Les Ulis). [Mathematics (Les
  Ulis)]. \MR{MR1933784 (2003m:32013)}

\bibitem{Sabb22}
\bysame, \emph{{Universal unfoldings of Laurent polynomials and tt*
  structures}}, {From Hodge theory to integrability and TQFT: tt*-geometry}
  (R.~Donagi and K.~Wendland, eds.), Proc. Symposia in Pure Math., vol.~78,
  American Math. Society, Providence, RI, 2008, pp.~1--29.

\bibitem{TA}
Atsushi Takahashi, \emph{{$tt\sp \ast$} geometry of rank two}, Int. Math. Res.
  Not. (2004), no.~22, 1099--1114. \MR{MR2041650 (2005b:53143)}

\bibitem{CV}
Claire Voisin, \emph{Hodge theory and complex algebraic geometry. {I}}, english
  ed., Cambridge Studies in Advanced Mathematics, vol.~76, Cambridge University
  Press, Cambridge, 2007, Translated from the French by Leila Schneps.
  \MR{MR2451566}

\end{thebibliography}
\end{document}